\documentclass{amsart}
\usepackage{amsthm,amssymb}
\newtheorem{theorem}{Theorem}

\newtheorem{corollary}{Corollary}
\newtheorem{lemma}{Lemma}
\newtheorem{example}{Example}

\newcommand{\Tu}{Tur\'{a}n }

\begin{document}
\title[Polynomials ]{Tur\'{a}n inequalities for three term recurrences with monotonic coefficients}

\author[I. Krasikov]{Ilia Krasikov}

\address{   Department of Mathematical Sciences,
            Brunel University,
            Uxbridge UB8 3PH United Kingdom}
\email{mastiik@brunel.ac.uk}


%\date{10.1.11}
%%%%%%%%%%%%%%%%%%%%%%%%%%%%%%%%%%%%%%%%%%%%%
%\newcommand{\majp}[1]{\succ\!\!_{#1}}
%\newcommand{\be}{\begin{equation}}
%\newcommand{\ee}{\end{equation}}
%\newcommand{\QED}{\end{proof}}
%\newtheorem{lemma}{Lemma}
%\newtheorem{theorem}{Theorem}

%-------------------- Title Page ----------------------------------

\begin{abstract}
We establish some new Tur\'{a}n's type inequalities for orthogonal polynomials defined by a three-term recurrence with monotonic coefficients. As a corollary we deduce asymptotic bounds on the extreme zeros of orthogonal polynomials with polynomially growing coefficients of the three-term recurrence.
\end{abstract}

%------------------------------------------------------------------------%
%                                                                        %
%         1. INTRODUCTION                                                %
%                                                                        %
%------------------------------------------------------------------------%
\maketitle
\section{Introduction}
\label{sec1}
Consider a family $\{p_i \}$ of orthogonal polynomials defined by the initial conditions
$p_{-1}=0,$  $p_0=1,$ and the three term recurrence
\begin{equation}
\label{grec}
\frac{a_k}{ c_k} \,  p_k
=(x-b_{k-1})p_{k-1}-a_{k-1}c_{k-1} \, p_{k-2} ,
\end{equation}
  where the coefficients $a_k$ and $c_k$ are strictly positive for all $k$ besides $a_0$ which will be convenient to set to zero.
  Thus the choice $c_k = 1$ or $c_k = a_k$  for all $k$ corresponds to the orthonormal and the monic normalization respectively. We will use boldface characters to distinguish orthonormal polynomials from these in a different normalization.

Although all the information about any member of the family $\{p_k\}$ is encoded in
(\ref{grec}), it turns out to be notoriously difficult to extract it from the recurrence.
One of few tools we possess today to deal with the problem is so-called Tur\'{a}n inequalities.
Usually they are written in the form
\begin{equation}
\label{clasturan}
p_k^2(x)-p_{k-1}(x)p_{k+1}(x)>0
  \end{equation}
  and are valid for many particular families of orthogonal polynomials including all classical ones as well as some other members of the Askey scheme. They play an important role
  in many applications, e.g. to recover the absolutely  continuous part of the corresponding orthogonality measure or to bound the extreme zeros (see e.g. \cite{szwarc1}, \cite{chenism}, \cite{dimit}, \cite{krturan}, \cite{mtn1}).

   In terms of the coefficients of the three-term recurrence very general conditions for the validity of (\ref{clasturan}) were given in \cite{szwarc}, provided that the support of the corresponding measure is finite or half-infinite, such that the polynomials can be normalized to one at some point. The symmetric case $b_i \equiv 0$ was considered in \cite{ismail1} and \cite{krturan}. It seems almost nothing is known in the general asymmetric case with the measure supported on the whole axis.

   In a sense, Tur\'{a}n inequalities can be viewed as an analogue of the Laguerre inequality $f'^2-f f'' >0.$ Since the last one is just the first member of an infinite family of inequalities discovered by Jensen \cite{jensen} (and rediscovered a number of times later, see e.g. \cite{patrick1}), it is worth trying to look for higher order generalizations of Tur\'{a}n inequalities. Some results in this direction were given in \cite{dimit} and \cite{krturan} and recently this question was raised again by P. Nevai \cite{nevai}.

   Another related set of inequalities comes from the Newton inequality stating that
   $ a_i^2-a_{i-1}a_{i+1} >0,$
   where $a_i$ are the coefficients of a real polynomial
   $$\sum_{i=0}^n a_i {n \choose i}  x^i$$
   with only real zeros \cite{dimit}, \cite{nicu}.

   It is not clear how far such an analogy can go. For example, an important difference between these and Tur\'{a}n inequalities is that in the last case normalization of $p_k$ plays a crucial role. In fact, both Laguerre and Newton type inequalities lead directly to higher order Tur\'{a}n inequalities only if the corresponding orthogonal polynomials have a generating function of a very special type \cite{dimit}, \cite{patrick1}, \cite{patrick2}.
   In this paper we explore these analogies to establish some new Tur\'{a}n inequalities for three term recurrence with monotonic coefficients both in the symmetric and in the general case. Thus, unless the coefficients of (\ref{grec}) are bounded, the corresponding measure has an unbounded support.

   As an application we give an upper bound on the largest zero of a polynomial satisfying (\ref{grec}) with the coefficients $a_k$ and $b_k$ of polynomial growth,
   $$a_k \sim k^r, \; \; b_k \sim  k^s.$$

   The following remarkable result for the symmetric case was obtained in \cite{mtn}, \cite{mtn2}:
   \begin{theorem}
   \label{themtn}
   Let $p_k$ be a family of monic symmetric polynomials given by
   $$p_{k}=x p_{k-1}-a_{k-1}^2 p_{k-2} , $$
   where
   $$a_k=c \, k^r \left(1+o \left(k^{-2/3} \right) \right), $$
   as $k \rightarrow \infty$ and $r >0.$ Let
   $x_{1,k} < x_{2,k} < \cdots  , x_{k,k}$
   be the zeros of $p_k.$ Then, for a fixed $j,$
   $$x_{k-j,k} =2 \, c \, k^r \left(1- r^{2/3} 6^{-1/3} i_j \, k^{-2/3}+o \left( k^{-2/3} \right) \right),$$
   where $i_1 < i_2 < \cdots $ are positive zeros of the Airy function.
      \end{theorem}
      The crucial point of the proof of this theorem is that the behaviour of a few largest zeros under an appropriate rescaling mimics that of the Hermite polynomials.
      It was also shown that at least for some particular families of symmetric polynomials
      and the extreme zeros similar results can be obtained via chain sequences \cite{chenism}. The last paper provides also very sharp bounds for the extreme zeros of certain asymmetric polynomials. This seems to be the only known result of this type for an infinite interval of orthogonality. In \cite{krturan} we gave some new Tur\'{a}n inequality for symmetric polynomials and used them to give a non-asymptotic version of Theorem \ref{themtn}, yet with an unavoidably weaker constant at the second order term, since the obtained bounds hold for any $k .$

    It turns out that to extend Tur\'{a}n inequalities to the general asymmetric case
    one has to impose rather severe constrains either on the coefficients of (\ref{grec}) or to restrict the rang of $x$ for which the inequality holds. In particular, most of the results of this paper deal with
 monotonic sequences $a_k$ and $b_k,$ where $a_k$ is strictly increasing. Moreover, without loss of generality we will assume that the sequence $b_k$ is nondecreasing. To justify this assumption it is enough to notice that the polynomials $q_k(x)=(-1)^k p_k(-x)$ satisfy the recurrence
$$\frac{a_k}{ c_k} \,  q_k
=(x+b_{k-1})q_{k-1}-a_{k-1}c_{k-1} \, q_{k-2} . $$

     It is important to stress that to apply the approach adopted in this paper for bounding the largest zero $x_{kk}$ (the least zero $x_{1k}$) it is enough to have a Tur\'{a}n inequality which holds for $x >x_{kk}$ (respectively $x <x_{1k}$).
   Here we will establish some inequalities of this type. In particular we use one such
  (rather technical) higher order inequality to give a version of Theorem \ref{themtn} in the general asymmetric case.

  Let us notice that for polynomially growing coefficients of (\ref{grec}),
   $a_i  \approx i^r , \; \; b_i \approx \gamma \, i^s ,$ the sought asymptotic depends on the three parameters $r,s,$ and $\gamma$ (seemingly more general case
   $a_i  \approx \gamma_1 i^r , \; \; b_i \approx \gamma_2 \, i^s ,$ is just a rescaling, see e.g. formulas (\ref{levshm}) and (\ref{levshM}) below).
   As we don't know asymptotics of the extreme zeros of any specific three-parametric family
   with polynomially growing coefficients one hardly can use the technics of \cite{chenism} or \cite{mtn}.
  On the other hand we conjecture that our method gives the correct order for the second term of the asymptotic yet with a weaker constant. Namely, we prove the following:
  \begin{theorem}
  \label{thmainzer}
  Let $$m= \frac{1}{\left( \frac{3}{2}\right)^{1/r}-1}-2 \,$$
  $$\rho= \frac{2}{3} \, \min\{1,r-s+1 \}, $$ and let
 $$a_i =i^r +o(i^{r- \rho}), \; \;
 b_i=\gamma \, i^s +o(i^{r -  \rho})  , $$
where  $r\ge 1$ and $0 \le s <r+1 $ are fixed numbers, and the constant $\gamma$ satisfies
$$\gamma \le
\left\{
\begin{array}{cc}
\frac{(m+2)^{r-s+1}}{2s} \, , & 0 <s <1,\\
&\\
\frac{(m+3)^{r-s+1}}{3s} \, , & 1 \le s < r+1.
\end{array}
\right.$$
Then for sufficiently large $k$ the largest zero of $p_k$ does not exceed
\begin{equation}
\label{eqmain}
\gamma k^s+k^r \left( 2-2^{-4/3} \delta^{2/3} \, k^{- \rho} \right)+
o \left( k^{r- \rho} \right) ,
\end{equation}
where $\delta$ is any fixed number satisfying
$$
\delta < \left\{
\begin{array}{cc}
 2r  ,& 0 \le s <r,\\
&\\
 (2+\gamma)r  ,& s=r,\\
& \\
\gamma \, s ,& r< s <r+1  .
\end{array}
\right.
$$
\label{thmain}
\end{theorem}

The restrictions on $\gamma$ and $\delta$ in the last theorem are definitely not best possible. However it seems that some constraints of that type are necessary.
Concerning the least zero of $p_k$ let us notice that the first term of its asymptotic (or, more precisely a lower bound)
is given by
$\gamma k^s-2k^r ,$ (see e.g. Theorem \ref{zerc} below).
 It seems plausible that the second order term may be quite different from $O(k^{r- \rho})$
  in general.

  The paper is organized as follows. In the next section we introduce some possible generalizations of Tur\'{a}n inequalities. Under appropriate conditions the suggested inequalities will be proved in section \ref{sec3}.  We also provide some examples with the Stieltjes-Wigert, Al-Salam-Carlitz and the Meixner-Pollaczek polynomials. Sections \ref{sec4} and \ref{sec5} deal with the proof of Theorem \ref{thmainzer}. It turns out that it is enough to prove the result for a suitably chosen test sequences $a_k=\alpha_k, \; b_k=\beta_k.$ Such sequences and an appropriate Tur\'{a}n type inequality will be given in section \ref{sec4}. We'll need some rather tedious calculations to verify that the chosen sequences $\{\alpha_k\}, \, \{\beta_k\}$ satisfy the inequality. Given a Tur\'{a}n type inequality, there is a quite straightforward way to derive the sought bounds on the extreme zeros, provided some technical conditions are fulfilled. This will be accomplished in section \ref{sec5}, thus proving Theorem \ref{thmainzer}.

  %%%%%%%%%%%%%%%%%%%%%%%%%%%%%%%%%%%%%%%%%%%%%%%%%%%%%%%%%%
\section{Tur\'{a}n inequalities. Preliminaries}
\label{sec2}
Consider the case of monic polynomials defined by
\begin{equation}
\label{monicrec}
q_{k+1}(x)=(x-b_k)q_k(x)+a_k^2 q_{k-1}(x); \; \; q_{-1}=0, \, q_0=1.
\end{equation}
First we notice that Tur\'{a}n's inequality
 \begin{equation}
\label{monicturan}
q_k^2(x)-\xi_k q_{k-1}(x) q_{k+1}(x) >0,
\end{equation}
is equivalent to the following one for polynomials in arbitrarily normalization defined by (\ref{grec}),
$$p_k^2(x)-\xi_k  \frac{a_{k+1}c_k}{a_k c_{k+1}} \, p_{k-1}(x) p_{k+1}(x). $$
Since
$$q_k^2(x)-q_{k-1}(x)q_{k+1}(x)=$$
$$(b_k-b_{k-1})x^{2k-1}+
\left(a_k^2-a_{k-1}^2-(b_k-b_{k-1})\left(b_k+2 \sum_{i=0}^{k-1}b_i \right) \right)x^{2k-2}+O(x^{2k-3}),$$
then (\ref{monicturan}) may hold (and in fact holds) with $\xi_k=1$ for all $x \in \mathbb{R}$ only if $b_k$ are constant and $a_k$ are nondecreasing.
Thus, if we want to deal with general asymmetric orthogonal polynomials, we have either to choose $\xi_k <1$ of to restrict $x$ to a subset of $\mathbb{R}$ or both.
On the other hand, for sufficiently small $\xi_k ,$ (\ref{monicturan}) will be fulfilled for all $x.$ Indeed, if $\xi_k <1$ then $q_k^2(x)-\xi_k q_{k-1}(x) q_{k+1}(x) >0$ in the neighbourhood of
$\pm \infty,$ as well as around the zeros of $q_k(x),$ as there $q_{k-1}(x)$ and $ q_{k+1}(x)$ have opposite signs. This simple argument can be presented in a quantitative form, in particular it is not difficult to show that one can take
$$\xi_k= \frac{4a_k^2}{4a_k^2+\max \left((x_{1k}-b_k)^2  ,(x_{kk}-b_k)^2  \right)},$$
and this, in a sense, is best possible. Since in many cases an interval containing the extreme zeros is known, this yields an explicit \Tu inequality.

We will need the following theorem due to Wendroff (see e.g. \cite[Th.2.10.1]{ismailbook}).
\begin{theorem}
\label{wendr}
Given two real sequences $\{a_i\}, \; i \ge N,  \; \ a_i >0,$ and $\{b_i\}, \; i \ge N,$
and two sequences of interlacing numbers $x_1<x_2< \ldots <x_N$ and $y_1<y_2< \ldots <y_{N-1},$
such that $x_i <y_i <x_{i+1}, \; i=1, \ldots ,N-1.$ There is a family of monic orthogonal polynomials $\{q_i(x)\}$ such that
$$q_N(x)=\prod_{i=1}^N (x-x_i), \; \; q_{N-1}(x)=\prod_{i=1}^{N_1} (x-y_i), $$
and
$$q_{k+1}(x)=(x-{\widetilde b}_k)q_k(x)- {{\widetilde a}_k}^{\,2} \, q_{k-1}(x)$$
and
${\widetilde a}_k=a_k, \; \; {\widetilde b}_k=b_k, $ for $k \ge N.$
\end{theorem}
\begin{theorem}
\label{prop1}
Let $\{q_n\}$ be a family of monic orthogonal polynomials satisfying
 (\ref{monicrec}) and let $x_{1n}<x_{2n}< \ldots < x_{nn}$ be the zeros of $q_n(x).$
Then for $k \ge 2$ inequality (\ref{monicturan}) holds for
\begin{equation}
\label{bestxi}
\xi_k= \left\{
\begin{array}{cc}
\frac{4a_k^2}{4a_k^2+(x_{kk}-b_k)^2} \, , & x> b_k-\sqrt{4a_k^2+(x_{kk}-b_k)^2}\, , \\
& \\
\frac{4a_k^2}{4a_k^2+(x_{1k}-b_k)^2} \, , & x< b_k+\sqrt{4a_k^2+(x_{1k}-b_k)^2}\, . \\
\end{array}
\right.
\end{equation}
Moreover, one can choose $\xi_k=1,$ provided either $x >b_k-2a_k$ and $  x_{kk} \ge b_k,$
or $x <b_k+2a_k$ and $  x_{1k} \le b_k.$

The result of (\ref{bestxi}) is the best possibe in the sense that for any fixed
 $k \ge 2,\; a_k >0, \; b_k, \; x_{1,k}, \; x_{kk}; \,$ $x_{1,k}< x_{kk},$ $b_k < x_{kk},$
 (resp. $b_k > x_{1k}$),
 and any
 $\xi_k >  \frac{4a_k^2}{4a_k^2+(x_{kk}-b_k)^2} \, , $ (resp. $\xi_k >  \frac{4a_k^2}{4a_k^2+(x_{1k}-b_k)^2} \,  $),
 there is a family of monic orthogonal polynomials and a point $x >x_{kk}$ (resp. $x <x_{1k}$ ) such that
 $$q_k^2(x)- \xi_k q_{k-1}(x)  q_{k+1}(x) <0.$$
\end{theorem}
\begin{proof}
We will give a proof for the case where the largest zero $x_{kk}$ is involved, the second one is similar.\\
We set
$t_k =t_k(x)= q_{k-1}(x)/q_k(x),$ and $h=b_k-x_{kk}.$
Using (\ref{monicrec})  we find
$$F(x,t_k)= \frac{q_k^2(x)-\xi_k q_{k-1}(x)  q_{k+1}(x)}{q_k^2(x)}=  \xi_k a_k^2 t_k^2-
\xi_k (x-b_k)t_k+1.$$
First notice that this quadratic
in $t_k$ is positive for
$$b_k -\frac{2a_k}{\sqrt{\xi_k}} <x< b_k +\frac{2a_k}{\sqrt{\xi_k}} \, .$$
Hence it will be enough to prove the claim for $x$ beyond this interval.
By
$x_{ik} < x_{i,k-1}<x_{i+1,k}$ we have
$$t_k = \frac{1}{x-x_{kk}} \, \prod_{i=1}^{k-1} \frac{x-x_{i,k-1}}{x-x_{i k}}
< \frac{1}{x-x_{kk}}\, , \; \; \; x> x_{kk}.$$
Hence for $x > x_{kk}$ we can set
$t_k= \frac{\epsilon}{x-x_{k,k}} \, ,$
where $ \epsilon =\epsilon(x) \in (0,1).$

For $\xi_k =\frac{4a_k^2}{4a_k^2+(x_{kk}-b_k)^2}= \frac{4a_k^2}{4a_k^2+h^2} \, $ we have
$$x_{kk}< b_k +\frac{2a_k}{\sqrt{\xi_k}} = b_k+\sqrt{4a_k^2+h^2} \le x,$$

and one finds
$$F(x,t_k)= \frac{4 a_k^4 \epsilon^2}{(x-x_{kk})^2 (4a_k^2+h^2)}-
\frac{4a_k^2 (x - b_k) \epsilon}{(x-x_{kk})(4a_k^2+h^2)} +1=$$

$$\frac{\left( 2a_k^2 \epsilon-(x+x_{kk})h \right)^2+
4a_k^2(x-x_{kk})^2 (1- \epsilon)}{(x-x_{kk})^2 (4a_k^2+h^2)} >0, $$
 and (\ref{bestxi}) follows.

Next we shall prove that for $x >b_k- 2a_k,$ that is for $h \ge 0,$ and $x_{kk} \ge b_k,$ one can take $\xi_k =1.$
Setting
$x =b_k+ \delta,$ where $\delta \ge 2a_k \, ,$ we obtain

$$F(x,t_k) =1- \frac{\epsilon (\delta^2+\delta h-a_k^2 \epsilon)}{(\delta+h)^2} >
1- \frac{\delta^2+\delta h-a_k^2 }{(\delta+h)^2} =\frac{a_k^2}{(\delta+h)^2}+ \frac{h}{\delta+h} \ge 0.$$

Let us show that the above choice of $\xi_k$ in (\ref{bestxi}) is the best possible, provided
$x_{kk} >b_k.$ By Wendroff's theorem for any fixed
$k,\; a_k >0, \; b_k, \; x_{1,k}< x_{kk},$
there is a corresponding orthogonal polynomial with these parameters
such that
the product
$$\prod_{i=1}^{k-1} \frac{x-x_{i,k-1}}{x-x_{i k}}  $$
is arbitrarily close to one for $x>x_{kk}$. Hence we can set
$t_k =\frac{\epsilon}{x-x_{kk}} \, $ with $\epsilon $ arbitrarily close to one.
Given a $\xi_k,$ we find
$$F(x,t_k)=1- \frac{\epsilon\left( (x-b_k)(x-x_{kk})-
a_k^2 \epsilon \right)}{(x-x_{kk})^2} \, \xi_k .$$
By the assumption $x_{kk}>b_k,$ hence $h<0$ and we choose $x=x_{kk}- \frac{2a_k^2 \epsilon}{h}\,  $ obtaining

$$F(x,t_k)=1- \frac{4a_k^2 \epsilon+h^2}{4a_k^2}\, \xi_k .$$
This implies $$\xi_k <\frac{4a_k^2}{4a_k^2 \epsilon+h^2} \, ,$$
assuming that $F(x,t_k)$ is positive
and the result follows.
\end{proof}
In connection with this theorem it's worth noticing that $x_{1k} \le b_{k-1} \le x_{kk}$ for any orthogonal polynomial. Indeed, let $J_k$ be the truncated Jacobi matrix (see
(\ref{jacmatrix}) below). Then for $0 \le i \le k-1$ we have
$$ x_{1k}=\inf_{||v||=1} (J_k v,v) \le (J_ke_i,e_i) =b_i \le \sup_{||v||=1} (J_k v,v)=x_{kk},$$
where $(.,.)$ and $||.||$ denote the scalar product and the norm in $\mathbb{C}^k$ and $e_i$ is a vector in $\mathbb{C}^k$ with zero entries except for $i$th which is equal to $1.$

Arguments similar to these above are readily applicable to higher order \Tu inequalities of the form
\begin{equation}
\label{highordtur}
p_k^2 +\sum_{i=1}^m \xi_k^{(i)} p_{k-i}p_{k+i} >0.
\end{equation}
Indeed, one can start with the usual \Tu inequality and choose sufficiently small $\xi_k^{(i)}$
for $i \ge 2.$ However this seems rather misleading since such a proof suggests very small values of $\xi_k^{(i)},$ whereas in fact they can grow exponentially e.g. for Hermite polynomials, as the following result of Jensen shows \cite{jensen}.

A higher order generalization of the Laguerre inequality
  $f'^2-f f'' >0$ which holds, in particular, for the functions of the so-called Polya-Laguerre class has the following form (for modern exposition see e.g. \cite{patrick1}, \cite{patrick2}).
  \begin{equation}
  \label{jen}
  \mathcal{L}_{2m}(f)=\frac{1}{2} \,\sum_{j=0}^m (-1)^{j+m} {2m \choose j} f^{(j)}f^{(2m-j)} \ge 0, \; \;, m=0,1,...
  \end{equation}
  In particular
$$\mathcal{L}_2(f)=f'^2-f f'',$$
  $$\mathcal{L}_4(f)=3f''^2-4f'f'''+f f^{(4)}.$$
    By analogue one can introduce
corresponding Tur\'{a}n type operators which will be considered in this paper. The first of
of them is just the standard \Tu one
\begin{equation}
  \label{t2}
 T_2 (p_k)=p_k^2-p_{k-1}p_{k+1},
\end{equation}
and the second is
\begin{equation}
  \label{t4}
T_4 (p_k)=3 p_k^2-4p_{k-1}p_{k+1}+ p_{k-2}p_{k+2}.
\end{equation}

To justify the form of two following operators $S_2$ and $S_4$ which will be considered in the sequel (the last will be treated for symmetric polynomials only ) we need some explanations.
It was noticed in \cite{ismail1} that in the monic normalization among the expressions of the form $p_k^2- \xi p_{k-1}p_{k+1},$
the polynomial $p_k^2-p_{k-1}p_{k+1}$ has the minimal possible degree.
Allowing an arbitrarily normalization defined by (\ref{grec}) let us consider the following polynomials:
\begin{equation}
\label{t2g}
p_k^2-\xi_k p_{k-1}p_{k+1},
\end{equation}
\begin{equation}
\label{t4g}
p_k^2-\mu_k p_{k-1}p_{k+1} + \nu_k p_{k-2}p_{k+2}.
\end{equation}
Notice that for $b_i \not\equiv 0,$
$$p_k(x)= \left(x^k-x^{k-1} \sum_{i=0}^{k-1}b_i+O(x^{k-2} ) \right)\prod_{i=1}^k \frac{c_i}{a_i} \,,  $$
whereas in the symmetric case $b_i \equiv 0,$
we have
$$p_k (x)= \left( x^k-x^{k-2} \sum_{i=1}^{k-1} a_i^2 +O(x^{k-4})\right)\prod_{i=1}^k \frac{c_i}{a_i} \, .$$
Simple calculations readily yield
that the degree of (\ref{t2g}) is minimal and equal to $2k-1$ if at least one of $b_i \ne 0, \; i=0,1,...,k-1,$
and $2k-2$ if $b_k  \equiv 0,$ for
$$\xi_k = \frac{c_k }{c_{k+1}} \cdot \frac{a_{k+1}}{a_k}\, .$$
The minimum of the degree of  (\ref{t4g}) is $2k-4$ in the symmetric case and is attained for
\begin{equation*}
\label{mu}
\mu_k  = \frac{c_k }{c_{k+1}} \cdot \frac{a_{k+1}(a_{k+1}^2+a_k^2-a_{k-1}^2-a_{k-2}^2)}{a_k (a_{k+1}^2-a_{k-2}^2)} \, ,
\end{equation*}
$$\nu_k = \frac{c_k c_{k-1}}{c_{k+2} c_{k+1}} \cdot
\frac{a_{k+2}a_{k+1} (a_k^2-a_{k-1}^2)}{a_k a_{k-1}(a_{k+1}^2-a_{k-2}^2)} \, .$$
This suggests to define the following two operators:
\begin{equation}
\label{s2def}
S_2 (p_k)=p_k^2- \frac{c_k }{c_{k+1}} \cdot \frac{a_{k+1}}{a_k} \, p_{k-1}p_{k+1} \, ,
\end{equation}
and
\begin{equation}
\label{s4def}
S_4 (p_k)=p_k^2- \frac{c_k }{c_{k+1}} \cdot \frac{a_{k+1}(a_{k+1}^2+a_k^2-a_{k-1}^2-a_{k-2}^2)}{a_k (a_{k+1}^2-a_{k-2}^2)} \,p_{k-1}p_{k+1}+
\end{equation}
$$\frac{c_k c_{k-1}}{c_{k+2} c_{k+1}} \cdot
\frac{a_{k+2}a_{k+1} (a_k^2-a_{k-1}^2)}{a_k a_{k-1}(a_{k+1}^2-a_{k-2}^2)} \,p_{k-2}p_{k+2}.$$
Notice that $S_2(p_k)=T_2(p_k)$ for monic polynomials.

Some additional motivation comes from the fact that the inequalities $S_2 >0$ and $S_4 >0$
are invariant with respect to normalization. Namely, for $p_k=d_k {\bf p}_k$ we have
$$S_2(p_k)=
d_k^2 S_2({\bf p}_k); \; \; S_4(p_k )=
d_k^2 S_4({\bf p}_k).$$
Thus for $S_2$ and $S_4$ it will be enough to consider the orthonormal case only.

Although we will not use it here, let us notice that for general asymmetric polynomials
the operator $S_4$ minimizing the degree of the output is defined by
$$\mu_k  = \frac{a_{k+1}c_k (b_{k+1}+b_k-b_{k-1}-b_{k-2})}{a_k c_{k+1}(b_{k+1}-b_{k-2})} \; ,$$
$$\nu_k  =
\frac{a_{k+2}a_{k+1}c_k c_{k-1}(b_{k}-b_{k-1})}{a_k a_{k-1}c_{k+2}c_{k+1}(b_{k+1}-b_{k-2})} \, .$$

%%%%%%%%%%%%%%%%%%%%%%%%%%%%%%%%%%%%%%%%%%%%%%%%%%%%%%%%%%%%%%%%%%%%%%%%
\section{Tur\'{a}n's inequalities. Results}
\label{sec3}
We start with with the usual Tur\'{a}n inequality $T_2 (p_k) > 0.$

\begin{theorem}
\label{turgenn} Suppose that
\begin{equation}
\label{uslc1}
4\left( a_{i+1} \frac{c_{i}}{c_{i+1}}-a_{i} \right)\left( a_{i}-a_{i-1} \frac{
c_{i-1}}{c_i} \right) > (b_{i}-b_{i-1})^2,
\end{equation}
and
\begin{equation}
\label{uslc2}
\frac{a_{i+1}}{a_i} > \frac{c_{i+1}}{c_i}, \; \; i=1, \ldots ,k;
\end{equation}
then
$$T_2 (p_k) > 0.$$
\end{theorem}
\begin{proof}
The claim follows from $T_2(p_0)=1, $ and the identity
\begin{equation}
\label{idenk} 4c_{k+1}\left(a_{k+2}c_{k+1}-a_{k+1}c_{k+2} \right)
\left( a_{k+2} T_2(p_{k+1})-a_k c_k c_{k+2} T_2(p_k)\right)=
\end{equation}
$$
\left\{2
(a_{k+2}c_{k+1}-a_{k+1}c_{k+2})p_{k+1}+(b_{k+1}-b_k)c_{k+1}c_{k+2}p_k
\right\}^2+
$$
$$
c_{k+1}c_{k+2}
\left\{4(a_{k+2}c_{k+1}-a_{k+1}c_{k+2})(a_{k+1}c_{k+1}-a_k
c_k)-(b_{k+1}-b_k)^2c_{k+1}c_{k+2} \right\}p_k^2 .
$$
\end{proof}
Finding an appropriate sequence $c_k$ in Theorem \ref{turgenn} may be far from obvious.
On the other hand, rather simple sufficient conditions can be obtained by a suitable choice of $c_k.$ For example, choosing $c_1=1$ and
\begin{equation}
\label{eqnormaliz}
c_k=2^{k-1} \prod_{j=1}^{k-1}\left(\frac{a_{j+1}}{a_j}+ \frac{a_j}{a_{j+1}}\right)^{-1},\, \; k=2,3,\ldots ,
\end{equation}
one finds
\begin{corollary}
\label{colloropt}
Suppose that $a_i$ is increasing.
Let
\begin{equation}
\label{eqsw1}
2(a_2^2-a_1^2)>(b_1-b_0)^2,
\end{equation}
and
\begin{equation}
\label{eqsw2}
a_k^{-2} (a_{k+1}^2-a_k^2)(a_k^2-a_{k-1}^2 )> (b_k-b_{k-1})^2, \; \; k=2,3, \ldots ,
\end{equation}
then $T_2(p_k) >0.$
\end{corollary}
For growing $a_k$ the choice of $c_k$ given by (\ref{eqnormaliz}) is close to the best possible.
Indeed, since both factors in the left hand side of (\ref{uslc1}) must be positive, this yields
$$\frac{a_i}{a_{i+1}}< \frac{c_i}{c_{i+1}} \, , \; \; i=1,2, \ldots , k; \; \; \;
\frac{a_{i-1}}{a_i}< \frac{c_i}{c_{i-1}} \,, \; \; i =2, \ldots ,k .$$
Hence,
$$4 \left( a_{i+1} \frac{c_{i}}{c_{i+1}}-a_{i} \right)\left(a_{i}-a_{i-1} \frac{
c_{i-1}}{c_i} \right) < 4a_{i+1}a_{i-1}\left( \frac{a_{i+1}}{a_i} -
\frac{a_i}{a_{i+1}} \right)\left(\frac{a_i}{a_{i-1}} -\frac{a_{i-1}}{a_i} \right) =$$
$$4 a_i^{-2} (a_{i+1}^2-a_i^2)(a_i^2-a_{i-1}^2), $$
for $i \ge 2.$ Similarly, for $i=1$ the left hand side of (\ref{uslc1}) is less than $4(a_2^2-a_1^2).$

In fact, one can say even more.
The following nice observation is not mine, but I was unable to find out who made it.
Basic facts on chain sequences can be found in \cite{chih} and \cite{ismailbook}.
\begin{theorem}
\label{chain}
The assumptions of Theorem \ref{turgenn} are fulfilled if and only if the sequence $a_i$ is increasing and
$$\frac{a_i^2(b_i-b_{i-1})^2}{4(a_{i+1}^2-a_i^2)(a_i^2-a_{i-1}^2)} \,, \; \; \; i=1,\ldots , k,$$
is a chain sequence.
\end{theorem}
\begin{proof}
 The sequence $a_i$  must be increasing as (\ref{uslc1}) and  (\ref{uslc2}) imply
$$\frac{a_{i+1}}{a_i} > \frac{c_{i}}{c_{i+1}} \, , \; \; \frac{a_{i+1}}{a_i} > \frac{c_{i+1}}{c_i} \, .$$
Defining
$$g_i=\frac{a_i}{a_{i+1}} \cdot \frac{\frac{c_i}{c_{i+1}}-\frac{a_i}{a_{i+1}}}{1-\frac{a_i^2}{a_{i+1}^2}} < 1,$$
turns (\ref{uslc1}) into
$$(1-g_{i-1})g_i \ge  \frac{a_i^2(b_i-b_{i-1})^2}{4(a_{i+1}^2-a_i^2)(a_i^2-a_{i-1}^2)} >0 . $$
Hence
$$\frac{a_i^2(b_i-b_{i-1})^2}{4(a_{i+1}^2-a_i^2)(a_i^2-a_{i-1}^2)}$$ is a chain sequence.

In the opposite direction, if it is a chain sequence then any smaller one is also a chain sequence.  Therefore one can find $g_i$ and then the corresponding $c_i.$
\end{proof}
Since the sequence $1/2, \, 1/4, \, 1/4, \, ...$ is a chain sequence Theorem \ref{chain} readily yields Corollary \ref{colloropt}. However there are chain sequences with terms greater than $1/4,$ (see e.g. \cite{chenism},\cite{ismail},\cite{szwarc2},\cite{szwarc3}). In particular, a positive constant sequence $\{g\}_1^n$ is a chain sequence iff
$0 <g \le \frac{1}{4}\, \cos^{-2} \frac{\pi}{n+2} \, ,$ \cite[Th. 7.2.6]{ismailbook}.
%%%%%%%%%%%%%%%%%%%%%%%%%%%%%%%%%%%%%%%%%%%%%%%%%%%%%%%%%
\begin{example} Choosing in (\ref{grec})
$$a_k=q^{-2k} \sqrt{q (1-q^k)}, \; \;\; b_k=q^{-2k-1}(1+q-q^{k+1}), \; \; 0 <q <1, $$
one obtains so called Stieltjes-Wigert polynomials $S_k (x;q).$
The sequence $a_k$ is increasing and (\ref{eqsw1}) is fulfilled by
$$2(a_2^2-a_1^2)-(b_1-b_0)^2 =q^{-7}(1-q^2)(2-q-2q^2+2q^3) >0.$$
For (\ref{eqsw2}) we obtain
$$\frac{(1-q^k)q^{4k+3}}{(1-q)^2}\left(a_k^{-2} (a_{k+1}^2-a_k^2)(a_k^2-a_{k-1}^2 )- (b_k-b_{k-1})^2 \right)= $$
$$(1-q^2)^2(1-q^k)(1+q+q^2-q^{k+1}) +q^{3k+3} >0. $$
Hence in the normalization defined by (\ref{eqnormaliz}) we have $T_2(S_k (x;q))>0.$
\end{example}
\begin{example} The orthonormal Al-Salam-Carlitz polynomials ${\bf V}_k^{(\alpha)}(x;q)$ are defined by
$$a_k= q^{-k} \sqrt{\alpha q(1-q^k)}\, ,\; \; b_k=(\alpha+1)q^{-k}, \; \; 0 <q <1, \; \; \alpha >0.$$
The sequence $a_k$ is increasing and one finds
$$2(a_2^2-a_1^2)-(b_1-b_0)^2 =q^{-3} (1-q)(2 \alpha-(1+\alpha^2)(q-q^2)). $$
This expression is positive for
\begin{equation}
\label{eqalsalwig}
\frac{1-\sqrt{1-(q-q^2)^2}}{q-q^2} < \alpha < \frac{1+\sqrt{1-(q-q^2)^2}}{q-q^2} \, .
\end{equation}
For (\ref{eqsw2}) we obtain
$$\frac{(1-q^k)q^{2k+1}}{(1-q)^2}\left(a_k^{-2} (a_{k+1}^2-a_k^2)(a_k^2-a_{k-1}^2 )- (b_k-b_{k-1})^2 \right)=  $$
$$
(\alpha-q)(1-\alpha q)(1-q^k)+\alpha q^{2k+1}>0,
$$
provided
$ q < \alpha <q^{-1}.$
The last inequality is ever stronger than (\ref{eqalsalwig}) and we conclude that for the
Al-Salam-Carlitz polynomials normalized by (\ref{eqnormaliz}) and $ q < \alpha <q^{-1},$
$T_2(V_k^{(\alpha)}(x;q)) >0. $
\end{example}

In the symmetric case and the monic normalization the following result was obtained in \cite{krturan}:
 \begin{theorem}
\label{tht4symmon}
Suppose that polynomials $p_i$ satisfy
\begin{equation}
\label{symmon}
p_i=x p_{i-1}- a_{i-1}^2 p_{i-2}, \; \; p_1=0, \; p_0=1.
\end{equation}
Suppose further that $a_i$ are strictly increasing and
\begin{equation}
\label{symt4usl}
a_{i-1}^2-3 a_i^2+3 a_{i+1}^2-a_{i+2}^2 \ge 0, \; \;1 \le i \le k-1 .
\end{equation}
Then for $k \ge 2,$
\begin{equation}
\label{turan22}
T_4(p_k) > 0.
\end{equation}
\end{theorem}
\begin{proof}
First, we find
$$T_4(p_2)=(-3a_1^2+3a_2^2-a_3^2)x^2+3a_1^4+a_1^2 a_3^2 >0.$$
We have the following directly checked identity
\begin{equation}
\label{uslt4sym}
T_4 (p_{k+1}) =
\end{equation}
$$a_{k-1}^2 T_4 (p_k)+(a_{k+2}^2+3 a_k^2 -4 a_{k-1}^2)T_2 (p_k) +(a_{k-1}^2-3 a_k^2+3 a_{k+1}^2-a_{k+2}^2)p_k^2 .$$
By Corollary \ref{colloropt} for symmetric polynomials $T_2(p_k) >0,$ provided $a_i$ are strictly increasing.
Now the result follows by the induction on $k$ and the convention $a_0=0.$
\end{proof}

In the next section we will establish the inequality $T_4({\bf p}_k)>0$ in the general case under rather technical conditions and a restriction on $x.$ For orthonormal symmetric polynomials defined by
$$a_i {\bf p}_i=x {\bf p}_{i-1}- a_{i-1} {\bf p}_{i-2}, $$
one can use
$$a_{k+2}a_{k+3} T_4({\bf p}_{k+1})=a_{k-1}a_{k+2}T_4({\bf p}_{k})+$$
$$\left(4a_{k+3}a_k- 4a_{k+2}a_{k-1}+a_{k+2}^2-a_k^2 \right)T_2({\bf p}_{k})+\mathcal{E}_k^{(1)}{\bf p}_{k+1}^2+\mathcal{E}_k^{(2)}{\bf p}_{k}^2 ,$$
where
$$\mathcal{E}_i^{(1)}=3a_{i+2}a_{i+3}-4a_{i+1}a_{i+3}+a_i a_{i+1},$$
$$\mathcal{E}_i^{(2)}=(4a_{i+3}-a_i)(a_{i+1}-a_i)-a_{i+2}(a_{i+2}-a_{i-1}).$$
However the initial condition
$$a_1^2 a_2^2 a_3 a_4 T_4({\bf p_2})= \mathcal{E}_1^{(1)} x^4+
\left[2a_4(2a_1^2 a_2+2a_3^2-3a_1^2 a_3)-a_1 a_2 (a_1^2+a_2^2+a_3^2) \right] x^2+$$
$$a_1^3 a_3(3a_1 a_4+a_2 a_3)>0,$$
looks rather complicated. Alternatively, one can use
$$a_1^2 a_2 T_4({\bf p_1})= (3a_2-4a_1) x^2+4a_1^3.$$
 This assumption requires $T_4({\bf p_1})=3{\bf p}_1^2-4 {\bf p}_2 >0,$ what is much stronger than the standard Tur\'{a}n inequality
$T_2({\bf p_1})={\bf p}_1^2-{\bf p}_2 >0.$

Now we will establish the inequality $S_2( {\bf p}_{k} ) \ge 0,$ that is the standard \Tu inequality $p_k^2-p_{k-1}p_{k+1} >0$ in the monic normalization.
To do this we have to relax the condition that it holds for all $x \in \mathbb{R}.$

The following simple fact was noticed by different authors (see e.g. \cite{dombn}, \cite{shi}).
\begin{lemma}
\label{thoneway}
$ S_2( {\bf p}_{k} ) \ge 0$ for $b_k-2a_k \le x \le b_k+2a_k.$
\end{lemma}
\begin{proof}
By (\ref{grec})
$${\bf p}_k^2-\frac{a_{k+1}}{a_k }\, {\bf p}_{k-1}{\bf p}_{k+1}=
{\bf p}_{k-1}^2+{\bf p}_k^2- \frac{x-b_k}{a_k} \, {\bf p}_{k-1}{\bf p}_k ,
$$
where discriminant is negative for $|x-b_k| <2a_k.$

\end{proof}
\begin{theorem}
 \label{turantwo}
 Suppose that $a_k$ and $b_k$ are nondecreasing, then
 $$ S_2( {\bf p}_{k} ) \ge 0$$
 for $x \ge b_k-2a_k, \; \; k=0,1,...$.
  \end{theorem}
\begin{proof}
We assume that either $a_k$ or $b_k,$ say, $a_k$ is strictly increasing. For the nondecreasing $a_k$ the required inequality follows by obvious limiting arguments.

The proof is by induction on $k.$
 By the previous lemma we may assume that $x > b_{k+1}+2a_{k+1}.$
Set $x=b_{k+1}+ 2 a_{k+1}+y^2 $ to impose this condition, and let also
$b_{k+1}=b_k+ \delta^2 .$

We have
$S_2({\bf p_0})=1,$ and
$$ a_1^2 S_2({\bf p_1})=(b_1-b_0)x+a_1^2+b_0^2-b_0 b_1 \ge 0, $$
for $x \ge b_1-2a_1.$
Choose
$$\lambda = \frac{a_k^2}{a_{k+1}^2}  \, \frac{x^2-(b_k+b_{k+1})x -
2a_k^2-2a_{k+1}^2+b_k b_{k+1}}{ (x-b_k+2a_k)(x-b_k-2a_k)} = $$
$$
\frac{a_k^2}{a_{k+1}^2}  \, \frac{y^4+(4a_{k+1}+\delta^2 )y^2+2(a_{k+1}^2-a_k^2 +\delta^2 a_{k+1})}{ (x-b_k+2a_k)(x-b_k-2a_k)}>0 .$$

 Since $b_{k+1}+2a_{k+1} > b_{k}+2a_{k}$ it will be enough to show that the quadratic
$$D({\bf p}_{k+1},{\bf p}_{k})=a_{k+1}^2(x-b_k-2a_k)(x-b_k+2a_k)\left( S_2( {\bf p}_{k+1} )- \lambda S_2( {\bf p}_{k} ) \right)=$$
$$
V {\bf p}_k^2-U {\bf p}_k {\bf p}_{k+1}+W {\bf p}_{k+1}^2,
$$
 is positive.\\
 We find
$$V=  a_{k+1}^2 \delta^4+(2a_{k+1}+y^2)(2a_{k+1}^2-a_k^2) \delta^2 +
(a_{k+1}^2-a_k^2)(y^4+ 4a_{k+1}y^2+4a_{k+1}^2-2a_k^2) , $$
$$U=2 a_{k+1} \left((a_{k+1}^2+a_k^2) \delta^2+(2a_{k+1}+y^2)(a_{k+1}^2-a_k^2) \right), $$
$$W=a_{k+1}^2 \left(\delta^4 +(2a_{k+1}+y^2) \delta^2 +2a_{k+1}^2 -2a_k^2 \right). $$
The discriminant of $D$ is
$$\Delta=-4a_{k+1}^2 c_{k+1}^2 (2a_{k+1}+2a_k+\delta^2+y^2)(2a_{k+1}-2a_k+\delta^2+y^2)
\left(a_{k+1}^2 \delta^4 + \right.$$
$$\left. (2a_{k+1}+y^2)(a_{k+1}^2-a_k^2) \delta^2+(a_{k+1}^2-a_k^2)^2 \right) < 0. $$
As $W >0$  we conclude that $D > 0.$ This completes the proof.
\end{proof}
In the symmetric case we have the following result.
\begin{theorem}
\label{symtur4}
Suppose $a_k$ are increasing, $b_k \equiv 0,$ then
\begin{equation}
\label{eqsymtur2}
S_2({\bf p_k}) > 0.
\end{equation}
If also for $i=3,...,k,$
\begin{equation}
\label{usl4}
R_i=
\end{equation}
$$ a_{i+1}^2 a_{i}^2 (a_{i-1}^2-a_{i-2}^2)- a_{i}^2 a_{i-2}^2(a_{i}^2-a_{i-2}^2)+
a_{i-2}^2 a_{i-3}^2 (a_{i}^2-a_{i-1}^2) \ge 0, $$
then
\begin{equation}
\label{eqsymtur4}
S_4({\bf p_k}) > 0.
\end{equation}
\end{theorem}
\begin{proof}
Inequality (\ref{eqsymtur2}) follows from $S_2( {\bf p}_{1} )=1$ and the identity
\begin{equation}
\label{ident2sym}
a_{k+1}^2 S_2({\bf p}_{k+1} ) -a_k^2  S_2({\bf p}_{k} )=
 \left(a_{k+1}^2-a_k^2 \right) {\bf p}_k^2 \ge 0.
 \end{equation}
Inequality (\ref{eqsymtur4}) follows from $S_4( {\bf p}_{2} )=1$ and the identity
\begin{equation}
\label{ident4sym}
a_{k+1}^2 (a_{k+2}^2-a_{k-1}^2)(a_k^2-a_{k-1}^2)S_4( {\bf p}_{k+1} ) =
\end{equation}
$$ a_{k-1}^2 (a_{k+1}^2-a_{k-2}^2) (a_{k+1}^2-a_k^2) S_4( {\bf p}_{k} )+
R_{k+1} S_2( {\bf p}_{k} ).$$
\end{proof}
In general, it is not easy to check the assumption $R_i \ge 0.$ We give the following sufficient conditions which restrict the growth of $a_k$ to
 $$ \sqrt{a_{k-1}a_{k+1}} \le a_k  \le \sqrt{\frac{a_{k-1}^2+a_{k+1}^2}{2}} \, .$$
\begin{lemma}
\label{uslr}
If for $i \le k, \; \; k \ge 3,$ the following conditions hold\\
({\it i})
$a_i$ is strictly increasing, \\
({\it ii})
$ a_{i+1}^2-2a_i^2+a_{i-1}^2 \ge 0, $\\
({\it iii})
$\frac{a_{i}}{a_{i+1}}  $ is nondecreasing,\\
Then
$R_k > 0.$
\end{lemma}
\begin{proof}
For $i \ge 3$ we rewrite the condition $R_i \ge 0$  as
\begin{equation}
\label{newformr}
\frac{a_{i}^2}{a_{i-1}^2} \cdot \frac{a_{i-1}^2}{a_{i-2}^2}
\ge \frac{a_{i}^2-a_{i-1}^2}{a_{i-1}^2-a_{i-2}^2} \cdot \frac{a_{i}^2-a_{i-3}^2}{a_{i+1}^2-a_{i-2}^2}\, .
\end{equation}
Since
$$\frac{a_{i}^2}{a_{i-1}^2}
-\frac{a_{i}^2-a_{i-1}^2}{a_{i-1}^2-a_{i-2}^2} = \frac{a_{i-1}^4-a_{i-2}^2 a_{i}^2}{a_{i-1}^2(a_{i-1}^2-a_{i-2}^2)}\, ,$$
we obtain that $R_i \ge 0$ if
$$\frac{a_{i-2}}{a_{i-1}} \le \frac{a_{i-1}}{a_{i}} \, , $$
that is if ({\it i}) holds,
and
$$ \frac{a_{i-1}^2}{a_{i-2}^2}
\ge \frac{a_{i}^2-a_{i-3}^2}{a_{i+1}^2-a_{i-2}^2}\, .$$
To prove the last inequality it is enough to notice that
$\frac{a_{i-1}^2}{a_{i-2}^2} \ge 1,$ whereas by ({\it ii}) the sequence $a_{i+1}^2-a_i^2$ is nondecreasing and
$$\frac{a_{i}^2-a_{i-3}^2}{a_{i+1}^2-a_{i-2}^2}-1 =- \, \frac{a_{i+1}^2-a_{i}^2-( a_{i-2}^2-a_{i-3}^2)}{a_{i+1}^2-a_{i-2}^2} < 0.$$
This completes the proof.
\end{proof}
It's worth noticing that a necessary condition for positivity of $R_i$
can be also stated in terms of chain sequences.
Namely, $R_i >0$ implies that
$$\frac{a^2_{i-2} (a^2_i-a^2_{i-1})(a^2_{i+1}a^2_{i-1}-a^2_i a^2_{i-2})}{a^2_{i-1}a^2_i(a^2_{i+1}-a^2_{i-2})^2} $$
is a chain sequence, provided
$a_i$ are increasing. Indeed, defining
$$g_i = \frac{a^2_{i-1}}{a^2_i} \cdot \frac{1-\frac{a^2_{i-2}a^2_i}{a^2_{i-1}a^2_{i+1}}}{1-\frac{a^2_{i-2}}{a^2_{i+1}}} <1,$$
one can rewrite $R_i$ as follows
$$R_i= \frac{a^2_{i+1}}{a^2_i a^2_{i-1}} \cdot
\frac{(a^2_{i-1}-a^2_{i-2})(a^2_i-a^2_{i-1})^2}{(a^2_{i+1}-a^2_{i-2})(1-g_{i-1})(1-g_i)^2} \cdot
(a^2_i g_i-a^2_{i-1}g_{i-1}).$$
Hence $R_i >0$ iff $a^2_i g_i-a^2_{i-1}g_{i-1}>0,$ or equivalently,
$$(1-g_{i-1})g_i >g_i- \frac{a^2_i }{a^2_{i-1}} \, g_i^2 = \frac{a^2_{i-2} (a^2_i-a^2_{i-1})(a^2_{i+1}a^2_{i-1}-a^2_i a^2_{i-2})}{a^2_{i-1}a^2_i(a^2_{i+1}-a^2_{i-2})^2}  \, .$$

\begin{example}
Let us illustrate the above results for the Meixner-Pollaczek polynomials
$P_k^{(\lambda )}(x; \phi)$ (see e.g. \cite{koek}). This is, probably, the simplest example of not necessarily symmetric polynomials supported on the whole axis.
In the orthonormal case they are defined for $ \lambda >0$ and $0 < \phi < \pi$ by (\ref{grec}) with
$$a_k = \frac{\sqrt{k(k+2 \lambda-1)}}{2 \sin \phi} \, ,$$
$$b_k=  (k+ \lambda)\cot \phi  .$$
Thus, $a_k$ and $b_k$ are strictly increasing provided $\phi <\pi/2.$
In the symmetric case $\phi=\frac{\pi}{2}$ the expression in (\ref{symt4usl}) vanishes. For $R_i$ defined by (\ref{usl4}) one finds
$$R_i=24(i+ \lambda-1) (i+ \lambda-2)(2i+ \lambda-3)>0.$$
Hence, in the monic normalization
$$T_2 ( P_k^{(\lambda )}(x; \frac{\pi}{2})) >0, \; \; \; T_4 ( P_k^{(\lambda )}(x; \frac{\pi}{2})) >0,$$
and in the orthonormal case
$$ S_4 ( {\bf P}_k^{(\lambda )}(x; \frac{\pi}{2})) >0.$$

For general orthonormal Meixner-Pollaczek polynomials  the condition (\ref{uslc2}) of Theorem \ref{turgenn} is obviously fulfilled. After some algebraic manipulations (\ref{uslc1}) becomes
$$(\sqrt{i(i+2 \lambda-1)}-\sqrt{(i-1)(i+2 \lambda-2)} \;)(\sqrt{(i+1)(i+2 \lambda)}-\sqrt{i(i+2 \lambda-1)} \; ) $$
$$>\cos^2 \phi .$$
Replacing $\cos \phi$ by one and solving the obtained inequality one concludes that
$T_2({\bf P}_k^{(\lambda )}(x; \phi))>0$ for $\lambda \ge 1/2.$ For $\lambda < 1/2$
validity of the inequality depends on $\phi.$\\
Finally, $S_2({\bf P}_k^{(\lambda )}(x; \phi))\ge 0$ for
$$x \ge (k+ \lambda)\cot \phi -\frac{\sqrt{k(k+2 \lambda-1)}}{2 \sin \phi} \, .$$
\end{example}

Let us notice that the identities used in the proofs show that $T_2, T_4, S_2,S_4$ can be written as a sum of squares. For example, one can easily check that in the symmetric case
and monic normalization
$$ T_2(p_k)= \sum_{i=0}^{k-1} (a_{i+1}^2-a_i^2)p_i^2 \prod_{j=i+1}^{k-1} a_j^2.$$
In the orthonormal case we have
$$S_2({\bf p}_{k} )=a_k^{-2} \sum_{i=0}^{k-1} (a_{i+1}^2-a_i^2) {\bf p}_i^2 .$$
We give one more expression of this type.
 \begin{lemma}
\label{cor2}
For symmetric polynomials, $b_k \equiv 0,$
\begin{equation}
\label{eqcor24}
S_4({\bf p}_{k} )=
\frac{1}{a_k^2 a_{k-1}^2(a_{k+1}^2-a_{k-2}^2)}
\sum_{i=1}^{k-2} \left((a_{i+1}^2-a_i^2)\mathcal{A}_k- (a_{k}^2-a_{k-1}^2)\mathcal{A}_{i+1} \right){\bf p}_i^2,
\end{equation}
where
$$ \mathcal{A}_i=a_i^4-a_{i-1}^4 +a_i^2 a_{i+1}^2-a_{i-1}^2a_{i-2}^2 .$$
\end{lemma}
\begin{proof}
Substituting
$$S_4( {\bf p}_{i} )= \frac{a_i^2-a_{i-1}^2}{a_i^2 a_{i-1}^2(a_{i+1}^2-a_{i-2}^2)} \, F_{i} ,$$
into (\ref{ident4sym}) yields
$$F_{k+1}-F_k = \frac{a_k^2 R_{k+1}}{(a_{k+1}^2-a_k^2)(a_k^2-a_{k-1}^2)} \,  S_2({\bf p}_k).$$
Using $F_2=\frac{a_2^2-a_1^2}{a_1^2a_2^2 a_3^2}$  and $a_0=0$ one finds,
$$F_k=F_2+\sum_{i=3}^{k}\frac{ R_i}{(a_{i}^2-a_{i-1}^2)(a_{i-1}^2-a_{i-2}^2)}
\, \sum_{j=0}^{i-2} (a_{j+1}^2-a_j^2) {\bf p}_j^2= $$
$$F_2+\sum_{j=0}^{k-2}(a_{j+1}^2-a_j^2) {\bf p}_j^2 \sum_{i= \max {\{3,j+2 \}}}^{k}\frac{ R_i}{(a_{i}^2-a_{i-1}^2)(a_{i-1}^2-a_{i-2}^2)}=  $$
$$ \sum_{j=1}^{k-2}(a_{j+1}^2-a_j^2) {\bf p}_j^2 \sum_{i= j+2}^{k}\frac{ R_i}{(a_{i}^2-a_{i-1}^2)(a_{i-1}^2-a_{i-2}^2)}$$
The innermost sum is transformed into telescoping sums and we obtain
$$\sum_{i=j+2}^{k}\frac{ R_i}{(a_{i}^2-a_{i-1}^2)(a_{i-1}^2-a_{i-2}^2)}=
 \sum_{i=j+2}^{k} (a_{i+1}^2+a_{i-1}^2-2a_{i-2}^2)+$$
$$
\sum_{i=j+2}^{k}  \left( \frac{a_{i+1}^2 a_{i-1}^2}{a_i^2-a_{i-1}^2}-
\frac{a_i^2 a_{i-2}^2}{a_{i-1}^2-a_{i-2}^2} \right)+
\sum_{i=j+2}^{k}  \left( \frac{a_{i-1}^4 }{a_i^2-a_{i-1}^2}-
\frac{ a_{i-2}^4}{a_{i-1}^2-a_{i-2}^2} \right)-
$$
$$
\sum_{i=j+2}^{k}  \left( \frac{a_{i}^2 a_{i-1}^2}{a_i^2-a_{i-1}^2}-
\frac{a_{i-1}^2 a_{i-2}^2}{a_{i-1}^2-a_{i-2}^2} \right)- \sum_{i=j+2}^{k}  \left( \frac{a_{i-1}^2 a_{i-2}^2}{a_i^2-a_{i-1}^2}-
\frac{a_{i-2}^2 a_{i-3}^2}{a_{i-1}^2-a_{i-2}^2} \right)=
$$
$$
a_k^2+a_{k-1}^2 +\frac{a_k^2 a_{k+1}^2-a_{k-1}^2 a_{k-2}^2}{a_k^2-a_{k-1}^2}- a_{j+1}^2-a_j^2-
\frac{a_{j+1}^2 a_{j+2}^2-a_{j}^2 a_{j-1}^2}{a_{j+1}^2-a_{j}^2} =
$$
$$
\frac{\mathcal{A}_k}{a_{k}^2-a_{k-1}^2} -\frac{\mathcal{A}_{j+1}}{a_{j+1}^2-a_{j}^2} \, ,
$$
and the result follows.
\end{proof}

%%%%%%%%%%%%%%%%%%%%%%%%%%%%%%%%%%%%%%

\section{A higher order Tur\'{a}n inequality for asymmetric case}
\label{sec4}
In this section we will establish, under some quite technical conditions, a higher order Tur\'{a}n inequality (Theorem \ref{t341}) which is tailored to deal with extreme zeros of the polynomials defined by a three term recurrence. We also show that the conditions of the theorem are fulfilled for some particular polynomially growing sequence $a_i =\alpha_i$ and $b_i =\beta_i$
which is used in the proof of Theorem \ref{thmainzer}. This will require some rather lengthly calculations.

We need the following claim which deals with arbitrarily coefficients $a_i>0$
and $b_i$  in recurrence (\ref{grec}).
\begin{lemma}
\label{center}
Let $x_{1k}$ and $x_{kk}$ be the least and the largest zero of $p_k,$ then
$$x_{kk} > \max_{i \le k}  \frac{b_{i-1}+b_{i-2}+\sqrt{4a_{i-1}^2+(b_{i-1}-b_{i-2})^2}}{2} \, > \max_{i \le k} b_{i} , $$

$$x_{1k} < \min_{i \le k}  \frac{b_{i-1}+b_{i-2}-\sqrt{4a_{i-1}^2+(b_{i-1}-b_{i-2})^2}}{2} \, < \min_{i \le k} b_{i} . $$
\end{lemma}
\begin{proof}
We prove the first inequality, the second one is similar. Consider the corresponding monic polynomials defined by
$$q_i(x)=(x-b_{i-1})q_{i-1}(x)-a_{i-1}^2 q_{i-2}(x) .$$
Let $x \ge x_{k,k},$ then all the polynomials $q_i(x), \; \; i \le k,$ are positive and therefore
$q_{i-1}(x) > \frac{q_i(x)}{x-b_{i-1}} \, .$  Hence
$$0\le q_i(x) <(x-b_{i-1})q_{i-1}(x)-a_{i-1}^2  \frac{q_{i-1}(x)}{x-b_{i-2}}=
q_{i-1}(x) \left( x-b_{i-1}-  \frac{a_{i-1}^2}{x-b_{i-2}}\right).$$
The expression in the brackets must be positive yielding that for any $x \ge x_{kk},$
$$x >  \frac{b_{i-1}+b_{i-2}+\sqrt{4a_{i-1}^2+(b_{i-1}-b_{i-2})^2}}{2} \, >  b_{i} , \; \; i \le k,$$
and the result follows.
\end{proof}
 Although the above result can be improved by iterating the arguments, it is probably rather weak in case of growing $a_i$ and $b_i .$ As far as we know there is no good lower (upper) bound on $x_{kk}$ ($x_{1k}$) in terms of the coefficients of three term recurrence.

Set $t_k=t_k(x)= \frac{{\bf p}_{k}}{{\bf p}_{k+1}} \, .$ Clearly, $t_k >0$ for $x > x_{k+1,k+1},$ and since $x_{k+1,k+1}>x_{k,k}$ the inequality $t_k  \ge 0$ for
$x >x_{k+1,k+1}$ implies $t_i >0$ for $x \ge x_{k+1,k+1}$ for all $i <k.$
\begin{theorem}
\label{t341}
Suppose that $a_k$ and $b_k$ are nondecreasing and for some $ j <k,$\\
$(i)$  $T_4 ({\bf p}_j) \ge 0$ for $x  \ge \max \{ x_{j+1,j+1} ,b_{j+1}+  a_{j+1} \} ,$\\
$(ii)$ for $ t \ge 0,$ the the following quadratics are nonnegative:
$$\mathcal{P}_i (t)= a_{i-1}  A_i + B_i t+ a_{i-1} C_i t^2, \; \; i=j+1, \ldots ,k,$$
where
$$A_i=3a_{i+1}a_{i+2}-4a_{i}a_{i+2}+a_{i-1} a_{i}, $$
  $$B_i=a_{i-1} (4a_{i+2}-a_{i-1})(b_{i}-b_{i-1})-a_{i}^2 (b_{i+1}-b_{i-2}), $$
 $$C_i=
 \left( a_i-  a_{i-1}+b_i-b_{i-1} \right)(b_{i+1}-b_{i-2})+(a_i-a_{i-1})(4a_{i+2}-a_{i-1})-$$
 $$a_{i+1}(a_{i+1}-a_{i-2}).$$
 $(iii)$ For $i=j, \ldots,k-1,$
 $$ a_i \ge b_{i+1}-b_i.$$

Then for $x > \max \{x_{k,k}, b_k+  a_k \} ,$
$$T_4 ({\bf p}_k) \ge 0. $$
\end{theorem}
\begin{proof}
First observe that $T_2 ({\bf p}_k) >0 $ for $t_k \ge 0.$
Indeed, $T_2 ({\bf p}_0)=1$ and
$$ {\bf p}_{k+1}^{-2} \left( a_{k+2} T_2 ({\bf p}_{k+1})-a_k T_2 ({\bf p}_k )\right)=
(a_{k+1}-a_k)t_k^2+(b_{k+1}-b_k)t_k+a_{k+2}-a_{k+1} \ge 0.$$
Now we will apply induction on $k.$ Suppose that $T_4({\bf p}_k) \ge 0$ for $x \ge x_{k,k}.$
By the assumption $x \ge \max \{ x_{j+1,j+1} ,b_{j+1}+ a_{j+1} \}$ we can set
$x=b_{k}+  a_k+ y^2.$

The induction step is given by the following identity
$$a_{k+2}a_{k+3}T_4({\bf p}_{k+1})-a_{k-1}a_{k+2}T_4({\bf p}_k) -\left((b_{k+2}-b_{k-1})y^2+\mu_{k+1} \right)T_2({\bf p}_k)=$$
$$
a_k^{-1}{\bf p}_{k+1}^{-2} \left(\mathcal{P}_{k+1}(t_k) +(a_{k+1}-a_k)(b_{k+2}-b_{k-1})y^2 t_k^2\right) \ge 0,
$$
where
$$\mu_{k+1}=a_{k+2}^2-a_k^2+4a_k a_{k+3}-4a_{k-1}a_{k+2}+( a_k+b_k-b_{k+1})(b_{k+2}-b_{k-1}) \ge 0. $$
 It is left to notice that to start the induction it is enough to assume $T_4({\bf p}_j) \ge 0$
 for $x \ge \max \{ x_{j+1,j+1} ,b_{j+1}+ a_{j+1} \},$ rather than $x \ge \max \{ x_{j,j} ,b_{j}+ a_{j} \}.$
 \end{proof}
 Let us notice that one can readily establish many slightly different sufficient conditions implying $T_4 ({\bf p}_k) \ge 0. $ Besides the choice of normalization, one can use $S_2$ instead $T_2$ in the proof together with various lower bounds on the largest zero which can be obtained from formula (\ref{levshM}) of the next section.

 The following lemma shows that the restrictions on the initial conditions imposed by $(i)$ of  Theorem \ref{t341} are fulfilled for $j=1,$ provided $a_2\ge \frac{4}{3} \, a_1 .$
 \begin{lemma}
\label{43h}
Let $a_2\ge \frac{4}{3} \, a_1 ,$
then $T_4 ({\bf p}_1) \ge 0$ for $x \ge b_1 .$
 \end{lemma}
\begin{proof}
The result follows from the explicit expression
$$a_1^2 a_2 T_4 ({\bf p}_1)=
(3a_2-4a_1)y^4+2  (3a_2-2a_1)(b_1-b_0)y^2+
4a_1^3 + 3 a_2(b_1-b_0 )^2,
$$
where we set $x=b_1+ y^2.$
 \end{proof}
 Let $r \ge 1, \;$\;  $0 \le s \le r+1,$ and let
\begin{equation}
\label{defm}
m=\frac{1}{\left(\frac{3}{2} \right)^{1/r}-1}-2 \ge 0.
\end{equation}
 To prove Theorem \ref{thmainzer} we shall verify the assumptions of Theorem \ref{t341} for the test sequences $a_i=\alpha_i$ and $b_i=\beta_i,$ defined by
\begin{equation}
\label{polseqa}
\alpha_0=0, \; \; \alpha_1=\frac{(m+2)^r}{2}, \; \; \alpha_i=(i+m)^r, \; \; i \ge 2;
\end{equation}
\begin{equation}
\label{polseqb}
 \beta_i=\gamma (i+m)^s, \; \; i \ge 0.
\end{equation}

Using the inequality
\begin{equation}
\label{inlog}
\frac{x}{1+ \frac{x}{2}} \le \ln(1+x) \le \frac{x}{1+ (\sqrt{2}-1) x} \, , \; \;  0 \le x \le 1,
\end{equation}
one finds
\begin{equation}
\label{ozenkam}
\frac{r}{\ln \frac{3}{2}}- \frac{5}{2} \le m \le \frac{r}{\ln \frac{3}{2}} - \sqrt{2}-1 .
\end{equation}

Obviously, for such a choice of $\alpha_i$ the assumption of Lemma \ref{43h} is fulfilled.
 Let us also notice the following easy to check properties of the sequence $\alpha_i:$
  for $i \ge 2,$
   \begin{equation}
  \label{usl1}
  \frac{\alpha_{i}}{\alpha_{i-1}} \;is \; nonincreasing,
 \end{equation}
 \begin{equation}
 \label{usl2}
  \alpha_{i}-2\alpha_{i-1}+\alpha_{i-2} \ge 0,
 \end{equation}
 \begin{equation}
 \label{usl3}
 \frac{\alpha_{i+1}}{\alpha_i} \le 2.
 \end{equation}

  First we show that for the chosen sequences of $\alpha_i, \beta_i$ the coefficients $A_i$ and $C_i$ of $\mathcal{P}_i (t)$ are nonnegative.
\begin{lemma}
\label{lemap}
Suppose that $a_i=\alpha_i$ and $b_i=\beta_i,$ then
for $i \ge 2,$
\begin{equation*}
\begin{array}{cc}
(i) & A_i  \ge 0,\\
&\\
(ii) & C_i  \ge (\alpha_i-\alpha_{i-1}+\beta_i- \beta_{i-1})(\beta_{i+1}-\beta_{i-2})\ge 0.
\end{array}
\end{equation*}
\end{lemma}
\begin{proof}
{\it (i)} Applying  $\alpha_{i-1} \ge 2\alpha_i-\alpha_{i+1}$ we obtain
$$3\alpha_{i+1}\alpha_{i+2}-4 \alpha_{i}\alpha_{i+2}+\alpha_{i-1}\alpha_{i} \ge$$
$$ 2(\alpha_{i+1}-\alpha_{i})(\alpha_{i+2}-\alpha_i)+
\alpha_{i+1}\alpha_{i+2}-2\alpha_{i}\alpha_{i+2}+\alpha_{i}\alpha_{i+1} = $$
$$2(\alpha_{i+1}-\alpha_{i})(\alpha_{i+2}-\alpha_i)+ \alpha_{i+1} (\alpha_{i+2}-2\alpha_{i+1}+\alpha_{i})+2(\alpha_{i+1}^2-\alpha_i \alpha_{i+2}) \ge$$
$$  2(\alpha_{i+1}-\alpha_{i})(\alpha_{i+2}-\alpha_i) \ge 4(\alpha_{i+1}-\alpha_{i})^2.$$

{\it (ii)} It is enough to show that
$$(\alpha_{i}-\alpha_{i-1})(4\alpha_{i+2}-\alpha_{i-1}) - \alpha_{i+1}(\alpha_{i+1}-\alpha_{i-2}) \ge 0.$$

Applying in turn $\alpha_{i-2} \ge 2\alpha_{i-1}-\alpha_i,$ and then $\alpha_{i+2} \ge 2\alpha_{i+1}-\alpha_i,$
 we obtain
$$
(\alpha_{i}-\alpha_{i-1})(4\alpha_{i+2}-\alpha_{i-1}) - \alpha_{i+1}(\alpha_{i+1}-\alpha_{i-2}) \ge
$$
$$
(\alpha_{i}-\alpha_{i-1})(4\alpha_{i+2}-\alpha_{i+1})-\alpha_{i-1}(\alpha_{i+1}+\alpha_i-2\alpha_{i-1}) \ge
$$
$$
(\alpha_{i}-\alpha_{i-1})(8\alpha_{i+1}-4\alpha_{i}-\alpha_{i-1})-
\alpha_{i+1}(\alpha_{i+1}+\alpha_i-2\alpha_{i-1}):=f.
$$
Since
$$\frac{\partial f }{\partial \alpha_{i-1}} =2\alpha_{i-1}+3\alpha_i-6\alpha_{i+1} < 0,$$
by using $\alpha_{i-1} \le \frac{\alpha_i^2}{\alpha_{i+1}},$ we get
$$f \ge  \frac{(\alpha_{i+1}-\alpha_{i})^2}{\alpha_{i+1}^2} \, (\alpha_i^2+5\alpha_i \alpha_{i+1}-\alpha_{i+1}^2) \ge 0,$$
for
$$\frac{\alpha_{i+1}}{\alpha_i} \le \frac{5+ \sqrt{29}}{2}\, . $$
\end{proof}

%%%%%%%%%%%%%%%%%%%%%%%%%%

To check the rest of the conditions of Theorem \ref{t341}, in particularly that
$\mathcal{P}_i (t)\ge 0$ for $t \ge 0$ we need two following elementary inequality.

Since for $0 \le x <y <1,$ the function $\frac{1-x^q}{1-y^q}$ is decreasing in $q,$
we have
for $q\ge 1,$
\begin{equation}
\label{eqner}
\frac{1-x^q}{1-y^q}\le \frac{1-x}{1-y}\, ,
\end{equation}
whereas for $0 \le q < 1,$
\begin{equation}
\label{bernoulli}
\frac{1-x^q}{1-y^q} \le \lim_{q \rightarrow 0}\frac{1-x^q}{1-y^q} =\frac{\log x}{\log y} \, .
\end{equation}
We also need the following version of Bernoulli's inequality
\begin{equation}
\label{ineqbern}
(1+x)^c \le 1+c x(1+x)^{c-1}, \; \; x \ge -1, \; \; c \ge 1,
\end{equation}
which holds by
$$x \, \frac{d}{dx} \left( (1+x)^c - c x(1+x)^{c-1} -1\right)= c(1-c)x^2(1-x)^{c-2} \le 0,$$
along with the usual Bernoulli inequality
\begin{equation}
\label{usualbern}
(1+x)^c \le 1+c \, x, \; \; x >-1, \; \; 0 \le c \le 1.
\end{equation}

We shall impose the following restriction on the parameter $\gamma$ to satisfy the condition $(iii)$ of Theorem \ref{t341}.
\begin{lemma}
\label{lemgamma}
Let
\begin{equation}
\label{gambound}
\gamma \le
\left\{
\begin{array}{cc}
\frac{ (m+2)^{r-s+1}}{2s} \, , & 0 <s <1,\\
&\\
\frac{ (m+3)^{r-s+1}}{3s} \, , & 1 \le s \le r+1.
\end{array}
\right.
\end{equation}
then
$ \alpha_i \ge \beta_{i+1}-\beta_i, \; \; i \ge 1,$
\end{lemma}
\begin{proof}
The sought bound on $\gamma $ is given by
$$ \min_{i \ge 1} \frac{\alpha_i}{(m+i+1)^s-(m+i)^s}=$$
$$
\min \left\{ \frac{(m+2)^{r-s}}{2 \left( \left( 1+ \frac{1}{m+2} \right)^s-1 \right)} \,, \; \min_{i \ge 2} \frac{(m+i)^{r-s}}{\left( 1+\frac{1}{m+i} \right)^s-1}, \right\} .$$
If  $0 <s <1,$ then
applying (\ref{usualbern}) we get
$$ \min \left\{ \frac{(m+2)^{r-s}}{2 \left( \left( 1+ \frac{1}{m+2} \right)^s-1 \right)} \,, \; \min_{i \ge 2} \frac{(m+i)^{r-s}}{\left( 1+\frac{1}{m+i} \right)^s-1}, \right\}
\ge $$
$$\min \left\{  \frac{(m+2)^{r-s+1}}{2s} \, , \; \min_{i \ge 2}  \frac{(m+i)^{r-s+1}}{s}  \right\} = \frac{(m+2)^{r-s+1}}{2s} \, .
$$
Similarly, on applying (\ref{ineqbern}) for  $1 \le s <r+1,$ we obtain
$$\gamma \le \min \left\{ \frac{(m+2)^r}{2s (m+3)^{s-1}}\, , \; \min_{i \ge 2}
 \frac{(m+i)^r}{s (m+i+1)^{s-1}}\right\} = \frac{(m+2)^r}{2s (m+3)^{s-1}} \, =$$
 $$  \frac{(m+3)^{r-s+1}}{2s} \left( \frac{m+2}{m+3}\right)^r= \frac{(m+3)^{r-s+1}}{3s}$$
\end{proof}

\begin{lemma}
\begin{equation}
\label{bdiffer}
\frac{\beta_{i+1}-\beta_{i-2}}{\beta_i-\beta_{i-1}} \le 3 \frac{\alpha_{i+1}}{ \alpha_i} \, , \end{equation}
for $i \ge 2$ if $s \ge 1,$ and for $i \ge 3$ if  $\;0 <s <1.$
\end{lemma}
\begin{proof}
Putting $n=m+i \ge i$ and applying (\ref{eqner}) we find for $s \ge 1,$
$$\frac{\beta_{i+1}-\beta_{i-2}}{\beta_i-\beta_{i-1}}=
\frac{1-\left(\frac{n-2}{n+1} \right)^s}{1-\left(\frac{n-1}{n} \right)^s} \, \left(\frac{n+1}{n} \right)^s \le \frac{3n}{n+1} \,\left(\frac{n+1}{n} \right)^s \le $$
$$3\left(\frac{n+1}{n} \right)^r =\frac{3\alpha_{i+1}}{\alpha_i}.$$
Similarly, in order to prove the claim for $0 <s <1,$ and $i \ge 3,$ by $r \ge 1$ it is enough to show that
$$
\frac{1-\left(\frac{n-2}{n+1} \right)^s}{1-\left(\frac{n-1}{n} \right)^s}\, \left(\frac{n+1}{n} \right)^s \le 3 \, \frac{n+1}{n} \, ,$$
or equivalently
$$\frac{1-\left(\frac{n-2}{n+1} \right)^s}{1-\left(\frac{n-1}{n} \right)^s}\, \left(\frac{n+1}{n} \right)^{s-1} \le 3. $$
If $s \ge \frac{1}{2}$ then it is an easy exercise to check that for $n \ge 3,$
$$ \frac{1-\left(\frac{n-2}{n+1} \right)^s}{1-\left(\frac{n-1}{n} \right)^s}\, \left(\frac{n+1}{n} \right)^{s-1}  \le
\frac{1- \sqrt{\frac{n-2}{n+1} }}{1-\sqrt{\frac{n-1}{n} }} \le 3. $$
If $0 <s <\frac{1}{2}$ then by (\ref{bernoulli}) and (\ref{inlog}), where it is enough to take $2/5$ instead of $\sqrt{2}-1,$
 we obtain
$$ \frac{1-\left(\frac{n-2}{n+1} \right)^s}{1-\left(\frac{n-1}{n} \right)^s}\, \left(\frac{n+1}{n} \right)^{s-1}  \le \frac{\ln \frac{n+1}{n-2}}{\ln \frac{n}{n-1} } \sqrt{\frac{n}{n+1}} \le \frac{30n-15}{10n-8}\sqrt{\frac{n}{n+1}} <3.$$
This completes the proof.
\end{proof}

\begin{lemma}
\label{power}  Suppose that $a_i=\alpha_i$ and $b_i=\beta_i,$ then $\mathcal{P}_i (t)\ge 0 $  for $t \ge 0,$
and $i \ge 1.$
\end{lemma}
\begin{proof}
By Lemmas \ref{43h} and \ref{lemap} we have $T_4 ({\bf p}_1) \ge 0$ for $x \ge b_1 ,$ and also $A_i  \ge 0, \; C_i \ge 0$ for $i \ge 2.$
We will show that $B_i \ge 0$ for $ i \ge 3,$ whereas for $i=2$ the discriminant
of $\mathcal{P}_2 (t) $ is negative, provided $B_2 <0.$

\noindent
$Case \; 1.$ $i \ge 3.$
It is enough to establish the inequality
\begin{equation}
\label{prover1}
\frac{\alpha_{i-1}(4\alpha_{i+2}-\alpha_{i-1})}{\alpha_i^2} \ge \frac{\beta_{i+1}-\beta_{i-2}}{\beta_i-\beta_{i-1}} \, ,
 \; \; i \ge 3.
\end{equation}
By (\ref{bdiffer}) we may replace this by
$$4 \alpha_{i-1}\alpha_{i+2}-\alpha_{i-1}^2 \ge 3\alpha_{i+1} \alpha_i.$$
Notice that by the definition of the sequence $\alpha_i ,$
$$\frac{\alpha_{i}}{\alpha_{i-1}} \le \frac{\alpha_3}{\alpha_2}=\frac{3}{2} \, .
$$
Finally, applying (\ref{usl2}) we can get rid in turn of $a_{i+2}$ and $a_{i+1}$ obtaining
$$4 \alpha_{i-1}\alpha_{i+2}-\alpha_{i-1}^2- 3\alpha_{i+1} \alpha_i \ge
a_{i+1}(8a_{i-1}-3a_i)-4a_{i-1}a_i-a_{i-1}^2 \ge$$
$$3(3a_{i-1}-2a_i)(a_i-a_{i-1}) \ge 0. $$

\noindent
$Case \; 2.$ $i = 2, \; \; s \ge 1.$ Put
$\beta_2- \beta_{1} =v, \; \; \beta_3- \beta_{0} = x v,$
where $x \le 3\alpha_3/\alpha_2\;$  by (\ref{bdiffer}). We may assume that $B_2 <0,$
otherwise there is nothing to prove. This yields
$$ \frac{1}{v} \, B_2=4\alpha_1 \alpha_3-\alpha_1^2-\alpha_2^2 x <0, $$
hence
$$x> \frac{4\alpha_1 \alpha_3-\alpha_1^2}{\alpha_2^2} = \frac{8\alpha_4-\alpha_2}{4 \alpha_2} \, .$$
Using $\alpha_1=\frac{\alpha_2}{2}$ and the bound on $C_2 \ge b_2-b_1$ which holds by Lemma \ref{lemap}, we estimate the discriminant of $\mathcal{P}_2 (t)$ as follows:
\begin{equation}
\label{disc}
\frac{1}{v^2} \, \left(B_2^2-4\alpha_{1}^2 A_2 C_2 \right)\le
\left(4\alpha_1\alpha_4-\alpha_1^2-\alpha_2^2  x \right)^2  -4\alpha_{1}^2
(3 \alpha_3 \alpha_4 -\alpha_2 \alpha_4 +\alpha_1 \alpha_2) x =
\end{equation}
$$
\frac{\alpha_2^2}{16} \left( 16\alpha_2^2 x^2 -48\alpha_3 \alpha_4 x +(8\alpha_4-\alpha_2)^2\right)
$$
Since $x \le \frac{3\alpha_3}{\alpha_2}$ the last quadratic is decreasing in $x.$
Indeed, for the derivative we get
$$32 x \, \alpha_2^2-48 \alpha_3 \alpha_4 \le  48\alpha_3 \alpha_2
( 2  - \frac{\alpha_4}{\alpha_2} \, ),$$
where
$$2  -  \frac{\alpha_4}{\alpha_2}= 2- \left( 2 \left( \frac{3}{2} \right)^{1/r} -1\right)^r
\le 0,$$
 easily follows by Jensen's inequality.

Plugging in the least possible value of $x= \frac{8\alpha_4-\alpha_2}{4 \alpha_2}$ into (\ref{disc}) we obtain
$$ \frac{16 \alpha_4 -2 \alpha_2}{\alpha_2 } \, (8 \alpha_2 \alpha_4-6  \alpha_3 \alpha_4-  \alpha_2^2).$$
Finally, applying (\ref{usl1}) and (\ref{usl2}) and
noticing that $\frac{\alpha_3}{\alpha_2} = \frac{3}{2}$ we get
$$8 \alpha_2 \alpha_4-6  \alpha_3 \alpha_4-  \alpha_2^2 \le
8 \alpha_3^2-6  \alpha_3 \alpha_4-  \alpha_2^2 \le
8 \alpha_3^2-6  \alpha_3 (2\alpha_3-\alpha_2) -  \alpha_2^2 \le$$
$$\alpha_2^2 \left(6  \, \frac{ \alpha_3}{\alpha_2} -4 \left( \frac{ \alpha_3}{\alpha_2}\right)^2 -1\right)=-1.$$
Hence, the discriminant is negative.

\noindent
$Case \; 3.$ $i = 2, \; \; 0 \le s < 1.$  Set as above $v=\beta_2-\beta_1.$
Then
$$\beta_3-\beta_0 \le 2v+ \beta_1-\beta_0. $$
By Bernoulli's inequality and Lemma \ref{lemgamma}
\begin{equation}
\label{b21}
v = \gamma \left((m+2)^s-(m+1)^s \right) \le \gamma s (m+1)^{s-1} \le \frac{(m+2)^r}{2} =\frac{\alpha_2}{2}  \,,
\end{equation}
\begin{equation}
\label{b10}
\beta_1-\beta_0 = \gamma \left((m+1)^s-m^s \right) \le \gamma s \, m^{s-1} \le \frac{(m+2)^r}{2}= \frac{\alpha_2}{2} \,.
\end{equation}
Assuming that $B_2$ is negative we obtain for the discriminant of $\mathcal{P}_2 (t),$
$$\frac{4}{\alpha_2^2 (\beta_3-\beta_0)} \, \left(B_2^2-4\alpha_{1}^2 A_2 C_2 \right)\le
4 \alpha_2^2 (\beta_3-\beta_0)-(\alpha_2+2v)(\alpha_2^2-8 \alpha_2 \alpha_4+6\alpha_3 \alpha_4) \le$$
$$ 4 \alpha_2^2 (2v+\beta_1-\beta_0)-(\alpha_2+2v)(\alpha_2^2-8 \alpha_2 \alpha_4+6\alpha_3 \alpha_4):=f.$$
Since
$$\frac{\partial f}{\partial v} =2(3 \alpha_2^2+8 \alpha_2 \alpha_4-6\alpha_3 \alpha_4)=
\alpha_2^2 \left( 6- \left( 2 \cdot 3^{1/r}-2^{1/r} \right)^r \right) > 0, $$
on applying (\ref{b21}) and (\ref{b10}) we have
$$f \le 4\alpha_2 (\alpha_2^2 +4 \alpha_2 \alpha_4 -3 \alpha_3 \alpha_4)=
\alpha_2 ^3 \left(4-\left( 2 \cdot 3^{1/r}-2^{1/r} \right)^r \right) <0.$$
This completes the proof.
\end{proof}
Thus we obtained
\begin{lemma}
\label{seqab} The conditions of
Theorem \ref{t341} are fulfilled for orthonormal polynomials defined by the sequences $\alpha_i$ and $\beta_i,$ provided $\gamma$ satisfies (\ref{gambound}).
\end{lemma}

%%%%%%%%%%%%%%%%%%%%%%%%%%%%%%%%%%%%%%%%%%%%%%%%%%%%%%%%%%%

\section{Zeros}
\label{sec5}
Let $\{ \bf{p}_k(x) \}, $ be a family of orthonormal
polynomials defined by the three term recurrence (\ref{grec}).
 It is well known that the zeros $x_{1,k}<x_{2,k}<...<x_{k,k}$ of $ p_k$
coincide with the eigenvalues of the corresponding Jacobi matrix
\begin{equation}
\label{jacmatrix}
\left(
\begin{array}{cccccc}
b_0 & a_1 & 0& 0& \hdots & 0\\
a_1 & b_1& a_2 & 0 & \hdots &0\\
0& a_2 & b_2 & a_3 & \hdots & 0\\
\vdots& & \ddots & \ddots & \ddots &\vdots\\
0& 0& \hdots &a_{k-2} & b_{k-2} & a_{k-1}\\
0& 0& \hdots &0 & a_{k-1} & b_{k-1}\\
\end{array}
\right)
\end{equation}
If we assume that the sequences $a_i$ and $b_i$ are nondecreasing then by
the Gershgorin theorem we have
\begin{equation}
\label{gershmax}
x_{k,k} \le \max \left(b_{k-2}+a_{k-2}+a_{k-1},b_{k-1}+a_{k-1}  \right).
\end{equation}
In many cases this inequality gives the main term of the corresponding asymptotics. To obtain a better bound requires much more efforts.

The Rayleigh quotient for the extreme eigenvalues of the
corresponding Jacobi matrix yields the
following elegant representation for the extreme zeros (see e.g.
\cite{freud}, \cite{hj}, \cite{leven}):

\begin{equation}
\label{levshm}
x_{1k}=\min \sum_{i=0}^{k-1} \left(b_i x_i^2-2 a_{i} x_i x_{i-1} \right) ,
\end{equation}
\begin{equation}
\label{levshM}
x_{kk}=\max \sum_{i=0}^{k-1}\left( b_i x_i^2+2 a_{i} x_i x_{i-1}\right) ,
\end{equation}
where the extrema are taken over all (or only over positive) $x_0,x_1,...,x_{k-1},$ subjected to
$\sum_{i=0}^{k-1} x_i^2 =1,$ and $x_{-1}=0.$

This implies that for perturbed orthonormal polynomials
$$ {\tilde a}_{k} {\tilde p}_{k}=(x-{\tilde b}_{k-1}){\tilde p}_{k-1}-{\tilde a}_{k-1} {\tilde p}_{k-2},$$
such that $|a_i -{\tilde a}_i| < \epsilon, $ and $ |b_i -{\tilde b}_i| < \delta$ for $ i \le k-1,$
we have
\begin{equation}
\label{pert}
|x_{k,k}-{\tilde x}_{k,k}| <2 \epsilon +\delta,
 \end{equation}
In particular we have the following claim.
\begin{lemma}
\label{lemreduction}
The statement of Theorem \ref{thmainzer} holds, provided it holds for the sequences
$\alpha_i$ and $\beta_i$ defined by (\ref{polseqa}) and (\ref{polseqb}) respectively.
\end{lemma}
\begin{proof}
Let $p_i$ be the family of orthogonal polynomials defined by
   $$a_i =i^r +o(i^{r- \rho}), \; \;
 b_i=c \, i^s +o(i^{r -  \rho})  , $$
 where  $\rho= \frac{2}{3} \, \min\{1,r-s+1 \}, \; \; 0 \le s < r+1, \; \; r\ge 1. $
 Let also ${\tilde p}_{i}$ be the family defined by the sequences $\alpha_i$ and $\beta_i.$
  Denote by $x_{k,k}$ and ${\tilde x}_{k,k}$ be the largest zero of $p_k$ and ${\tilde p}_{k}$
  respectively.
 One readily finds
 $$|a_1- \alpha_1|= O(1),$$
 $$|a_i- \alpha_i|= o(i^{r- \rho}), \; \; i \ge 2,$$
 $$|b_i- \beta_i|= O(i^{s-1})+o(i^{r -  \rho}) = o(i^{r -  \rho}), \; \; i \ge 0.$$
 Hence, for sufficiently large $k$ we have
 $$|x_{k,k}-{\tilde x}_{k,k}| = o(k^{r -  \rho}).$$
\end{proof}
%%%%%%%%%%%%%%%%%%%%%%%%%%%%%%%%%%%%%%%%%%%%%%%%%%%%%%%
Tur\'{a}n inequalities readily give bounds on the extreme zeros.
Consider the function $t=t_k(x)=p_k/p_{k+1}.$
It consists of $k+2$ branches $\mathcal{B}_0,...,\mathcal{B}_{k+1}$ separated by the zeros $x_{1,k+1}<...<x_{k+1,{k+1}}$ of $p_{k+1}.$
Observe that $\mathcal{B}_0$ changes from $0$ to $-\infty$ on $(-\infty,x_{1,k}),$
whereas $\mathcal{B}_{k+1}$ changes from $\infty$ to $0$ on $(x_{k+1,k+1} , \infty).$
Suppose we are given with a Tur\'{a}n inequality which is valid for a subset $ \mathcal{I} \subset \mathbb{R}.$ Using ({\ref{grec}) we can rewrite it as a quadratic in $t,$
$$
\mathcal{T}(x;t)=K_2(x) t^2- K_1(x) t+K_0 (x) >0.
$$
Let $f(x)$  be a function intersecting
$\mathcal{B}_0$ (or $\mathcal{B}_{k+1}$) at a point $y \in \mathcal{I}.$ Then
$y < x_{1,k}$ (or $y > x_{k+1,k+1}$ ) and $y$ must be among the solutions of inequality
$$\mathcal{T}(x;f)=K_2(x) f^2(x)- K_1 (x) f(x)+K_0 (x) >0. $$
Clearly, the most natural choice for $f$ is $f= K_1 /2 K_2,$ provided it is continuous in a sufficiently large vicinity of the sought zero. The following claim illustrate this approach in the simplest case, where the Tur\'{a}n inequality of Theorem \ref{turgenn} is used.
\begin{theorem}
\label{zerc}
Suppose that $ a_i$ and $b_i$ satisfy (\ref{uslc1}) and (\ref{uslc2}). Then
\begin{equation}
\label{ozkor1}
b_{k-1}- 2\sqrt{\frac{a_k a_{k-1}c_{k-1}}{c_{k}}}<x_{1,k} <x_{k,k} < b_{k-1}+ 2\sqrt{\frac{a_k a_{k-1}c_{k-1}}{c_{k}}} \, ;
\end{equation}
and if $a_i$ and $b_i$ satisfy (\ref{eqsw1}) and (\ref{eqsw2}) then
\begin{equation}
\label{ozkor11}
b_{k-1}-\sqrt{2(a_{k-1}^2+a_k^2)} <x_{1,k} <x_{k,k} < b_{k-1}+\sqrt{2(a_{k-1}^2+a_k^2)} .
\end{equation}
\end{theorem}
\begin{proof}
Theorem \ref{turgenn} gives
$$\mathcal{T}(x,t)=t^2- \frac{x-b_k}{a_k c_k}\, t +\frac{a_{k+1}}{a_k c_k c_{k+1}} >0.$$
Choosing $f(x)= \frac{x-b_k}{2a_k c_k} \, ,$ that clearly intersects both branches $\mathcal{B}_{0}$ and $\mathcal{B}_{k+1},$ yields
$$ 4a_k^2 c_{k}^2 \mathcal{T}(x;f)=-(x-b_k)^2 + \frac{4a_k a_{k+1}c_k}{c_{k+1}} >0,$$
and (\ref{ozkor1}) follows. Choosing now $c_k$ defined by (\ref{eqnormaliz}) and applying Corollary \ref{colloropt} we obtain (\ref{ozkor11}).
\end{proof}

Form now on we assume that the polynomials we are dealing with are orthonormal.

To prove Theorem \ref{thmain} we will apply Tur\'{a}n inequality of Theorem \ref{t341}. First, we need a slightly stronger result than that of Theorem \ref{zerc} under more restrictive conditions. In particular, it requires $\frac{a_{k+1}}{a_k} <\frac{2}{\sqrt{3}} \, .$
\begin{lemma}
\label{vspom}
Suppose that
$$\frac{4a_k^2-3a_{k+1}^2}{2a_{k+1}} \ge b_{k+1}-b_k. $$
Then
$$x_{k,k} < b_{k-2}+2a_{k-2}. $$
\end{lemma}
\begin{proof}
Theorem \ref{turantwo} states that $S_2({\bf p}_k) >0$ for $x \ge b_k-2a_k.$
Expressing ${\bf p}_{k-1}$ and ${\bf p}_k$ through ${\bf p}_{k+1}$ and ${\bf p}_{k+2},$
and setting $\tau={\bf p}_{k+1}/{\bf p}_{k+2}$ we obtain
$$ a_k^2 a_{k+1}^2 {\bf p}_{k+2}^{-2} S_2({\bf p}_k)= L_2 \tau^2-L_1 \tau+L_0,$$
where
$$L_2=
a_{k+1}^4-a_{k+1}^2 b_k b_{k+1}+a_k^2 b_{k+1}^2+(a_{k+1}^2 b_k-2a_k^2 b_{k+1}+a_{k+1}^2 b_{k+1})x -(a_{k+1}^2 -a_k^2)x^2 ,$$
$$L_1 =a_{k+2}\left(a_{k+1}^2 b_k-2a_k^2 b_{k+1}-(a_{k+1}^2-2a_k^2)x \right), $$
$$L_0 =a_k^2 a_{k+2}^2.$$
Choose $f=L_1/2L_2.$ Notice that $L_1$ and $L_2$ have no common factors as their resultant in $x$ is
$$a_{k+1}^4 a_{k+2}^2 \left((2a_k^2-a_{k+1}^2)^2-a_k^2(b_{k+1}-b_k)^2 \right) \ge $$
$$
a_{k+1}^4 a_{k+2}^2 \left((2a_k^2-a_{k+1}^2)^2-\frac{(4a_k^2-3a_{k+1}^2)^2}{4} \right) =$$
$$\frac{a_{k+1}^6 a_{k+2}^2}{4}\, (8a_k^2-5 a_{k+1}^2) >\frac{a_{k+1}^6 a_{k+2}^2}{2}\, (4a_k^2-3 a_{k+1}^2) \ge 0.$$
We find
$$L_2 f^2-L_1 f+L_0 = \frac{a_{k+1}^4 a_{k+2}^2}{4 L_2} \,
(x-b_k+2 a_k)(b_k+2a_k-x).$$
Since $b_{k+1}<x_{k+2,k+2} <b_{k+1}+2a_{k+1}$ by Lemmas \ref{center} and \ref{zerc}, it is left to check that $L_2(x)$ does not vanish in the interval $[b_{k+1}, b_{k+1}+2a_{k+1})$ and that $f$ intersects $\mathcal{B}_{k+2}.$
Indeed,
$$L_2(b_{k+1})= a_{k+1}^4 >0,$$
$$a_{k+1}^{-2} L_2(b_{k+1}+2a_{k+1})=4a_k^2-3a_{k+1}^2-2a_{k+1}( b_{k+1}-b_k)^2 \ge 0, $$
by the assumption.
On the other hand
$$L_1( b_{k+1})=-a_{k+1}^2 a_{k+2} (b_{k+1}-b_k) <0,$$
$$L_1( 3b_{k+1}-2b_k)=2a_{k+2}^2 (4 a_{k}^2-3 a_{k+1}^2)(b_{k+1}-b_k) >0,$$
where
$$3b_{k+1}-2b_k \le b_{k+1}+ \frac{4a_k^2-3a_{k+1}^2}{a_{k+1}}<b_{k+1}+a_k. $$
Thus, $L_1 (x) >0$ for $x \ge b_{k+1}+a_k$ and therefore $f(x) >0$ for
$b_{k+1}+a_k \le x < \xi,$ where $\xi$ is the largest zero of $L_2.$
Hence $f$ intersects $\mathcal{B}_{k+2}$ as $f \rightarrow \infty ,$
that is when $x$ approaches $\xi.$
\end{proof}

The explicit form of
$$\mathcal{T}(x,t)=a_{k-1} a_k a_{k+2} {\bf p}_{k+1}^{-2}T_4 ({\bf p}_k)$$ is
$$
\mathcal{T}(x,t)=K_2 t^2+K_1 t+ K_0,
$$
where
\begin{equation}
\label{kkk}
K_2= a_k^2 a_{k+1}+3a_{k-1} a_k a_{k+2}-a_{k+1}  b_{k-1} b_k+a_{k+1}(b_k+b_{k-1})x-a_{k+1} x^2,
\end{equation}
$$
K_1=4a_{k-1}a_{k+2}b_k-a_{k+1}^2 b_{k-1}+a_k^2 b_{k+1}-b_{k-1}b_k b_{k+1}+$$
$$(a_{k+1}^2-a_k^2-4a_{k-1}a_{k+2}+b_{k-1}b_k +b_{k-1}b_{k+1}+b_k b_{k+1})x-
$$
$$(b_{k-1}+b_k+b_{k+1})x^2+x^3, $$
$$
K_0=-a_{k+1}(b_{k-1}b_{k+1}-4a_{k-1}a_{k+2}-(b_{k+1}+ b_{k-1})x+x^2).
$$
In what follows we choose
\begin{equation}
\label{definf}
f=f(x)=-K_1/2K_2 .
\end{equation}
The following two lemmas show that it is continuous and intersects the required branch.
\begin{lemma}
\label{leminterval}
Suppose that a sequence $a_i$ satisfies (\ref{usl1}) and (\ref{usl2}). Then the interval $[b_{k-1},b_{k-1}+2a_{k-1}]$ lies between the zeros of $K_2,$ and
the function $f(x)$ is continuous on it.
\end{lemma}
\begin{proof}
One calculates
$$K_2 (b_{k-1})=a_k (a_k a_{k+1}+3a_{k-1}a_{k+2} )>0,$$
$$ K_2 (b_{k-1}+2a_{k-1})=a_k^2 a_{k+1} +3a_{k-1}a_k a_{k+2}-4a_{k-1}^2 a_{k+1}   +2a_{k-1}a_{k+1}(b_k-b_{k-1}),$$
To show that $ K_2 (b_{k-1}+2a_{k-1}) >0$ we first apply $a_{k+2} \ge 2a_{k+1}-a_k,$
yielding
$$a_k^2 a_{k+1} +3a_{k-1}a_k a_{k+2}-4a_{k-1}^2 a_{k+1} >a_k^2 a_{k+1}+6
a_{k-1}a_{k}a_{k+1}-4a_{k-1}^2a_{k+1}-3  a_{k-1}a_k^2.$$
The derivative of the last expression in $a_{k+1}$ is positive. Replacing $a_{k+1}$ by $2a_k-a_{k-1}$ we get that it is not less than
$$2(a_k-a_{k-1})(a_k^2+5a_k a_{k-1}-2a_{k-1}^2) >0.$$
\end{proof}
Now we will consider the sequences $\alpha_i$ and $\beta_i$ defined by (\ref{polseqa}) and (\ref{polseqb}) respectively. We also set as above $n=n(i)=m+i$ and assume that
$n$ is large enough to justify all approximations below.

First we notice that since $\beta_{i+1}-\beta_i=o(i^r)$ the assumption of Lemma \ref{vspom} is fulfilled for sufficiently large $k$ and therefore
\begin{lemma}
\label{vspomab}
For the sequences $\alpha_i$ and  $\beta_i$  and sufficiently large $k $
$$x_{k,k} < \beta_{k-2}+2\alpha_{k-2}, $$
\end{lemma}

\begin{lemma}
For sufficiently large $k$ the function $f(x)$ intersects $\mathcal{B}_{k+1}$ on \\ $(\beta_{k-1},x_{k+1,k+1}],$ provided $\gamma$ satisfies (\ref{gambound}).
\end{lemma}
\begin{proof}
By Lemmas \ref{leminterval} and \ref{vspomab} we have
$x_{k+1,k+1}<\xi,$
where is the largest zero $\xi$ of $K_2$
Notice that $K_1(x)$ has 3 real zeros $\eta_1 \le \beta_k <\eta_2 < \eta_3.$
Indeed,
$$K_1( -\infty)=-\infty,\; \; K_1( \infty)=\infty ,$$
$$K_1 (\beta_k)=\alpha_{k+1}^2(\beta_k-\beta_{k-1})+\alpha_k^2(\beta_{k+1}-\beta_k)\ge 0, $$
$$K_1 (\beta_k+\alpha_k)=\alpha_k (\alpha_{k+1}^2-4\alpha_{k-1}\alpha_{k+2})+(\beta_k-\beta_{k-1})
\left(\alpha_{k+1}^2+\alpha_k^2-\alpha_k(\beta_{k+1}-\beta_k) \right)=$$
$$-3n^{3r} +O\left(n^{\max\{3r-1,2s-2 \}}\right)<0. $$
Now we check that
$$\xi=\frac{\beta_{k-1}+\beta_k+ \tau}{2} \, ,$$
where
$$\tau=\sqrt{\frac{4\alpha_k}{\alpha_{k+1}} (\alpha_k \alpha_{k+1}+3\alpha_{k-1}\alpha_{k+2})+(\beta_k-\beta_{k-1})^2} \, ,$$
is less than the largest zero of $K_1.$
Then $\lim _{x \rightarrow \xi^{(-)}} =\infty $ and $f$ intersects $\mathcal{B}_{k+1}$ before $\xi.$
Indeed, we calculate
$$2 \alpha_{k+1} K_1(\xi)=(\alpha_{k+1}^3+3\alpha_{k-1}\alpha_k \alpha_{k+2}-4\alpha_{k-1}\alpha_{k+1}\alpha_{k+2}) \gamma+$$
$$
+\alpha_{k+1}(\alpha_{k+1}^2+4\alpha_{k-1}\alpha_{k+2})(\beta_k-\beta_{k-1})-
3\alpha_{k-1}\alpha_k \alpha_{k+2}(2 \beta_{k+1}-\beta_k-\beta_{k-1})= $$
$$
\left(-2\gamma r n^{3r-1} - 4c s n^{3r+s-1}\right)\left(1 +O(n^{-1}) \right) <0.
$$
It is left to verify that $\xi \le \beta_k+2\alpha_k.$
Using (\ref{usl2}) one finds
$$\alpha_k^{-1} K_2 (\beta_k+2\alpha_k)= 3\alpha_{k-1}\alpha_{k+2}-3\alpha_k \alpha_{k+1}-2\alpha_{k+1}(\beta_k-\beta_{k-1})\le $$
$$3\alpha_{k-1}\alpha_{k+1} \left(\frac{\alpha_{k+2}}{\alpha_{k+1}}-\frac{\alpha_{k}}{\alpha_{k-1}} \right) <0. $$
This completes the proof.
\end{proof}
Thus intersection of $f$ and $\mathcal{B}_{k+1}$ occurs at a point belonging to the set\\
$\{x:G(x)=4K_0 K_2-K_1^2 >0\}.$ Here $G(x)$ is a rather complicated polynomial of degree 6,
\begin{equation}
\label{eqg}
G(x)=-x^6+2( \beta_{k+1}+\beta_k+\beta_{k-1})x^5 +...
\end{equation}
\begin{lemma}
For sufficiently large $n$ the equation $G(x)=0$ has precisely two real zeros in the region $ 0 \le s< r+1.$
\end{lemma}
\begin{proof}
For the discriminant $Dis_x G$ Mathematica gives
$$Res_x G =-2^{16}(1-n^{-2})^{3r}(1+2n^{-1})^{3r} n^{16r-8} V(n),$$
where
$$V(n) = n^{14} \sum_{i=0}^7  \gamma^{2i} s^{2i} h_i \, n^{2i(s-r)},$$
$$h_0=64 \cdot 6^6  r^8 \left( 1+O(n^{-1}) \right), $$
$$h_1=-64 \cdot 6^6   r^6 \left( 1+O(n^{-1}) \right),$$
$$h_2=4 \cdot 6^7   r^4  \left( 1+O(n^{-1}) \right),$$
$$h_3=-4 \cdot 6^6   r^2  \left( 1 +O(n^{-1}) \right),$$
$$h_4=9 \cdot 6^4     \left( 1+O(n^{-1}) \right),$$
$$h_5=3915  n^{-2}  \left( 1+O(n^{-1}) \right),$$
$$h_6=388  n^{-4}  \left( 1+O(n^{-1}) \right),$$
$$h_7=12  n^{-6}  \left( 1+O(n^{-1}) \right).$$
Thus, for $s<r$ the sign of $V(n)$ is determined by the sign of $h_0,$
whereas for $r<s <r+1$ by the sign of $h_4.$
Hence, $V(n)>0$ and does not change the sign for $s \ne r.$
For $r=s$ one finds
$$ V(n)=4 \cdot 6^6 (4-\gamma^2)^4  r^8 n^{14r}\left( 1+O(n^{-1}) \right) >0,$$
provided $\gamma \ne 2.$
Finally, for $r=s \ne 1,$ and $\gamma=2$ calculations give
$$ V(n)=16 \cdot 12^6  r^8 n^{14r-4}\left( 1+O(n^{-1} \right) >0,$$
and for $r=s=1,$ and $\gamma=2$
$$ V(n)=6^9  (2n^2+2 n -1)n^{14r-10}>0.$$
Thus, in the region $1 \le r \le s <r+1$ the number of real roots of $G(x)$ does not depend on the choice of the parameters for sufficiently large $n.$
Choosing $s=r=1, \; \gamma=2$ we get the following test equation
$$g=-x^6+12n x^5-6(8n^2-2n+1)x^4+4(16n^3-24n^2+14n+1)x^3+$$
$$3(64n^3-60n^2+12n-3)x^2 +12(16n^3-22n^2+3n+1)x+4(16n^3-33n^2+14n-1)=0$$
The discriminant of $g$ in $x$ is
$$2^{21} \cdot 3^9 n^3(n^2-1)^3 (n+2)^3 (2n^2+2n-1) \ne 0 $$
for $ n >1.$
Choosing $n=2$ we obtain that the number of real zeros of $G$ is the same as that of
$$-x^6+24x^5-174x^4+244x^3+879x^2+564x+69=0.$$
The last equation has just two real zeros, numerically $x_1 \approx -0.26$ and $x_2 \approx 6.6.$
\end{proof}

%%%%%%%%%%%%%%%%%%%%%%%%%%%%%%%%%%%%%%%

Now Theorem \ref{thmainzer} readily follows by obvious limiting arguments from the next lemma together with Lemma \ref{lemreduction}.
\begin{lemma}
\label{lemmain1} Suppose $\gamma$ satisfies (\ref{gambound}).
Then for sufficiently large $k$ the largest zero of $p_k$ defined by the sequences $\alpha_i$ and $\beta_i$ does not exceed
$$\gamma n^s+n^r \left( 2-2^{-4/3} \delta^{2/3} \, n^{- \frac{2}{3} \min\{1,r+1-s \}} \right),$$
where $\delta$ is any fixed number satisfying
$$
\delta< \left\{
\begin{array}{cc}
 2r  ,& 0 \le s <r,\\
&\\
 (2+\gamma)r  ,& s=r,\\
& \\
\gamma s ,& r< s <r+\frac{1}{2}.
\end{array}
\right.
$$
\end{lemma}
\begin{proof}
Let $\xi$ be the largest zero of the equation
$$G(x)=4K_0 K_2-K_1^2=0. $$
It is enough to show that
$$\xi < \gamma n^s+n^r \left( 2-2^{-4/3} \delta^{2/3} \, n^{- \frac{2}{3} \min\{1,r+1-s \}} \right).$$
First we check that
$G(\beta_k+\alpha_k) >0.$ Then any $x>\beta_k+\alpha_k$ such that $G(x)<0$ is an upper bound on $\xi.$
Indeed, calculations yield
$$ G(\beta_k+\alpha_k)= n^{6r}( 27 -4\gamma s (3+5r-3s)n^{s-r-2}+
  2\gamma^2 s^2 n^{2s-2r-2} \, -$$
  $$ 4\gamma^3 r s^3 n^{3s-3r-4} -\gamma^4 s^4 n^{4s-4r-4})\left( 1+O(n^{-1}) \right)>0.$$
  Now the result follows by calculating for an appropriate $x$ the leading term of the expansion of
  $$\frac{n^{2-6r} G(x)}{4(\nu^3-1)}, \; \; \nu <1,$$
  which turns out to be positive and is equal to
  $$
  \left\{
\begin{array}{ccc}
4 r^2 ,& x=\gamma n^s+n^r \left(2- 2^{-4/3}\left(\frac{2r}{ n} \right)^{2/3}\nu \right), & 0 \le s <r,\\
&&\\
 (2+\gamma)^2 r^2  , & x=n^r \left(2+\gamma -  2^{-4/3} \left(\frac{(2+\gamma)r}{ n} \right)^{2/3} \nu \right),  & s=r,\\
&&\\
 \gamma^2 s^2 n^{2s-2r} , & x=\gamma n^s+n^r \left(1- 2^{-4/3} \left(\frac{\gamma s}{ n^{r+1-s}} \right)^{2/3} \nu \right), & r< s <r+\frac{1}{2} \, .
\end{array}
\right.
  $$
  This completes the proof.
  \end{proof}
  {\bf Acknowledgement}
  I am grateful to Mourad E. H. Ismail for a very fruitful discussion.

%%%%%%%%%%%%%%%%%%%%%%%%%%%%%%%%%%%%%%%%%%%%%%%%%%%%%%%%%

%************************* References ***********************************

\end{document}